\documentclass[11pt]{amsart}
\usepackage[english]{babel}
\usepackage[top=1.5in, bottom=1.5in, left=1.2in, right=1.1in]{geometry}

\usepackage{amsmath,amsthm,amsfonts,amssymb,soul}
\usepackage{bm}
\usepackage{epsfig}
\usepackage{cite}
\usepackage{enumitem}
\usepackage{lipsum}
\usepackage{color}
\usepackage[svg]{xcolor}
\usepackage{cite}
\usepackage{varioref,graphicx,epstopdf,subfigure}

\usepackage{verbatim}

 \catcode`\@=12

\newtheorem{thm}{Theorem}[section]
\newtheorem{prop}[thm]{Proposition}
\newtheorem{Lemma}[thm]{Lemma}

\newtheorem{Remark}[thm]{Remark}

\usepackage{txfonts}
\def\R{{\mathbb R}}
\def\N{{\NN}}
\def\Z{{\mathbb Z}}

\def\G{{\Gamma}}

\def\NN{{\mathbb N}}

\def\MM{{\mathcal M}}
\def\Mm{\Gamma}

\def\Xx{{\mathcal{X}}}
\def\cdwell{c_{\it dw}}
\def\cduff{c_{\it D}}
\def\Rdwell{R_{\it dw}}
\def\Rduff{R_{\it D}}

\def\Mdwell{M_{\it dw}}
\def\Mduff{M_{\it D}}

\newcommand{\dist}{{\rm dist}\,}
\newcommand{\diam}{{\rm diam}\,}

\newlength{\bibitemsep}
\newlength{\bibparskip}
\let\oldthebibliography\thebibliography
\renewcommand\thebibliography[1]{%
  \oldthebibliography{#1}%
  \setlength{\parskip}{\bibitemsep}%
  \setlength{\itemsep}{\bibparskip}%
}

\fontfamily{lmtt}\selectfont
\begin{document}
\title{Prescribed energy connecting orbits for gradient systems}
\date{January 21, 2019}

\author{Francesca Alessio}
\address{Dipartimento di Ingegneria Industriale e Scienze Matematiche, 
Universit\`a Politecnica delle Marche,
Via Brecce Bianche, I-60131 Ancona, Italy.}
\email{f.g.alessio@univpm.it}

\author{Piero Montecchiari}
\address{Dipartimento di Ingegneria Civile, Edile e Architettura, 
Universit\`a Politecnica delle Marche,
Via Brecce Bianche, I-60131 Ancona, Italy.}
\email{p.montecchiari@univpm.it}

\author{Andres Zuniga}
\address{CEREMADE, UMR CNRS ${\rm n}^{\circ}$~7534, Universit\'e Paris-Dauphine, PSL Research University, Place de Lattre de Tassigny, 75775 Paris Cedex 16, France} 
\email{zuniga@ceremade.dauphine.fr}

\begin{abstract} We are concerned with conservative systems $\ddot q=\nabla V(q)$, $q\in\R^{N}$ for a general class of potentials $V\in C^1(\R^N)$. Assuming that a given sublevel set $\{V\leq c\}$ splits in the disjoint union of two closed subsets $\mathcal{V}^{c}_{-}$ and $\mathcal{V}^{c}_{+}$, for some $c\in\R$, we establish the existence of bounded solutions $q_{c}$ to the above system with energy equal to $-c$ whose trajectories connect $\mathcal{V}^{c}_{-}$ and $\mathcal{V}^{c}_{+}$. The solutions are obtained through an energy constrained variational method, whenever mild coerciveness properties are present in the problem. 
The {\sl connecting orbits} are classified into brake, heteroclinic or homoclinic type, depending on the behavior of $\nabla V$ on $\partial \mathcal{V}^{c}_{\pm}$.
Next, we illustrate applications of the existence result to double-well potentials $V$,  and for potentials associated to systems of duffing type and of multiple-pendulum type. In each of the above cases we prove some convergence results of the family of solutions $(q_{c})$. 
\end{abstract}

\maketitle
{\bf 2010 Mathematics Classification.} 34C25, 34C37, 49J40, 49J45.\\

 \par
	 \hskip 54pt  homoclinic orbits, heteroclinic orbits, convergence of solutions.

\medskip

\section{Introduction}

In the present paper we are concerned with second order conservative systems 
\begin{equation}\label{eqn:gradSystem}
\ddot q=\nabla V(q),
\end{equation}
where potentials $V\in C^1(\R^N)$ are considered, for any dimension $N\geq 2$.\medskip 

We study the existence of particular solutions to \eqref{eqn:gradSystem}, for a class of potentials $V$ for which there exists some value $c\in\R$ so that the sublevel set 
$$\mathcal{V}^{c}:=\{x\in\R^N: V(x)\leq c\},$$ is the union of two disjoint subsets. More precisely, we assume that for some $c\in\R$,\smallskip
\begin{enumerate}[label=$(\mathbf{V}^{c})$,leftmargin=23pt,labelsep=.5em]
\item\label{enum:Vc}
There exist $\mathcal{V}^{c}_-,\, \mathcal{V}^c_+\subset\R^N$ closed sets, such that $\mathcal{V}^{c}=\mathcal{V}^{c}_-\cup \mathcal{V}^{c}_+$ and $\dist(\mathcal{V}^{c}_-,\mathcal{V}^{c}_+)>0$,
\end{enumerate}\smallskip
where $\dist(A,B):=\inf\{|x-y|: x\in A,\, y\in B\}$ refers to the Euclidean distance from a set $A\subset\R^{N}$ to a set $B\subset\R^N$.
\bigskip

Provided~\ref{enum:Vc} holds, we look for  {\sl bounded solutions} $q$ of \eqref{eqn:gradSystem} on $\R$ with {\em prescribed mechanical energy} at level $-c$ 
\begin{equation}\label{eqn:energyc}
E_{q}(t):=\tfrac12|\dot q(t)|^{2}-V(q(t))=-c,\;\;\hbox{ for all }\;t\in\R,
\end{equation} 
which in addition {\sl connect} the sets $\mathcal{V}^{c}_-$ and $\mathcal{V}^{c}_+$:
\begin{equation}\label{eqn:condbrake}
\inf_{t\in\R}\dist(q(t),\mathcal{V}^{c}_-)=\inf_{t\in\R}\dist(q(t),\mathcal{V}^{c}_+)=0,
\end{equation}
where $\dist(x,A):=\inf\{|x-y|:y\in A\}$  denotes the distance from a point $x\in\R^N$ to a set $A\subset\R^N$.
\smallskip

To better describe which kind of solutions of \eqref{eqn:gradSystem} satisfying \eqref{eqn:energyc} and \eqref{eqn:condbrake} one can get, it is better to make some simple qualitative reasoning.  
 Note that condition \eqref{eqn:condbrake} imposes $\inf_{t\in\R}\dist(q(t),\mathcal{V}^{c}_{\pm})=0$; this is true if either the solution $q$ touches one (or both) of $\mathcal{V}^{c}_{\pm}$ in a point, or if it accumulates $\mathcal{V}^{c}_{\pm}$ at infinity, that is to say, 
\begin{equation}\label{eqn:accatinfty}
\liminf_{t\to -\infty}\dist(q(t), \mathcal{V}^{c}_{\pm})=0,\;\;\text{ or }\;\;\liminf_{t\to +\infty}\dist(q(t), \mathcal{V}^{c}_{\pm})=0.
\end{equation}
In the first case, there exists a time $t_{0}$ such that $q(t_{0})\in \mathcal{V}^{c}_{\pm}$. In this situation we say that $q(t_{0})$ is a {\sl contact point} between the trajectory $q$ and $\mathcal{V}^{c}_{\pm}$, and that $t_{0}$ is a {\sl contact time}. Let us note right away that if $t_{0}$ is a {\sl contact} time, since $V(q(t_0))\leq c$, then the energy condition \eqref{eqn:energyc} imposes that $V(q(t_0))= c$ and $\dot q(t_0)=0$. Hence $t_0$ is a {\sl turning} time, i.e., $q$ is symmetric with respect to $t_0$. From this we recover that $q$ has at most two contact points.\smallskip\par\noindent

The {\sl connecting solutions} between $\mathcal{V}^{c}_{\pm}$ can therefore be classified into three types, corresponding to the different number of {\sl contact points} they exhibit. Precisely, we have
\begin{enumerate}[label=(\Roman*), itemsep=.25\baselineskip, leftmargin=3em]
\item\label{enum:classification1} Two {\sl contact points}:  In this case the solution $q$ has one contact point $q(\sigma)$ with  $\mathcal{V}^{c}_-$ and one contact point $q(\tau)$ with  $\mathcal{V}^{c}_+$. We can assume that $\sigma<\tau$ (by reflecting the time if necessary), and that the interval $(\sigma,\tau)$ does not contain other {\sl contact times}. Since the solution is symmetric with respect to both $\sigma$ and $\tau$, it then follows that it has to be periodic, with period $2(\tau-\sigma)$. The solution oscillates back and forth in the configuration space along the arc $q([\sigma,\tau])$, and verifies $V(q(t))>c$ for any $t\in (\sigma,\tau)$. This solution is said to be of {\sl brake orbit} type~(see
\cite{[Saif]}, \cite{[W]}). Let us remark that a {\sl brake orbit} solution has {\em only one} contact point with each set
$\mathcal{V}^{c}_{\pm}$. 
\item\label{enum:classification2} One {\sl contact point}:  In this case the solution $q$ is symmetric with respect to the (unique) {\sl contact} time $\sigma$, resulting that $V(q(t))>c$ for any $t\in\R\setminus\{\sigma\}$ and $q(\sigma)\in\mathcal{V}^{c}_{\pm}$. Moreover $\liminf_{t\to\pm\infty}\dist(q(t), \mathcal{V}^{c}_{\mp})=0$. These solutions are said to be of {\sl homoclinic type}.
\item\label{enum:classification3} No {\sl contact points}: In this last case the absence of contact times implies that $V(q(t))>c$ for any $t\in\R$, being that $\liminf_{t\to-\infty}\dist(q(t),\mathcal{V}^{c}_{\pm})=0$ and $\liminf_{t\to+\infty}\dist(q(t), \mathcal{V}^{c}_{\mp})=0$. These solutions are said to be of {\sl heteroclinic type}.
\end{enumerate}

A great amount of work regards the existence and multiplicity of brake orbits when $c$ is regular for $V$, and the set $\{V\geq c\}$ is non-empty and bounded; see~\cite{[1],[2],[3], [4], [5], [GGP1],[GGP2], [6], [7], [8]}.
 
 A unified approach for the study of general connecting solutions was first made via variational arguments in~\cite{alessio2013stationary} for systems of Allen-Cahn type equations, where the author already builds solutions in the PDE setting analogous to the ones of heteroclinic type, homoclinic type and brake type solutions (cf.~\cite[Theorem 1.2]{alessio2013stationary} for details, and also~\cite{[AlM3bump],[AlMalmost], [AlMbrake], [AlM-NLS], alessio2016brake} for related results and techniques).

Concerning the ODE case,  the problem of existence of {\sl connecting orbits} of~\eqref{eqn:gradSystem} and their classification into heteroclinic, homoclinic and periodic type has been recently studied in~\cite{antonopoulos2016minimizers}, and subsequently in~\cite{fusco2017existence} (see also \cite{fusco2018existence}) for potentials $V\in C^2(\R^N)$ that in addition to $(\mathbf{V}^{c})$ satisfy $\partial \{x\in\R^N:V(x)>c\}$ is compact.

\smallskip
Our approach to the problem is variational and is an adaptation of the arguments developed in~\cite{alessio2013stationary, alessio2016brake} to the  ODE setting. We work on the admissible class
\begin{eqnarray}\label{defMc}
\Gamma_c:=\{q\in H^{1}_{loc}(\R,\R^N): &{\rm (i)}&V(q(t))\geq c\hbox{ for all }t\in\R,\hbox{ and }\\
& {\rm (ii)}&\liminf_{t\to-\infty}\dist(q(t),\mathcal{V}^{c}_-)=\liminf_{t\to+\infty}\dist(q(t),\mathcal{V}^{c}_+)=0\},\notag
\end{eqnarray}
and we look for minimizers in $\Gamma_c$ of the Lagrangian functional 
\begin{equation}\label{eqn:defHc}
J_c(q):=\int_{-\infty}^{+\infty}\tfrac12|\dot q(t)|^{2}+(V(q(t))-c)\, dt.
\end{equation}

Note that  the set $\Gamma_c$ of admissible functions is defined via {\rm(i)} and {\rm(ii)}.  Condition {\rm(i)} constitutes an energy constraint, in that, the function $V(q(t))-c$ is non-negative over $\R$, so the functional $J_{c}$ is well defined and bounded from below on $\Gamma_c$. If $q$ is a minimizer of $J_{c}$ on $\Gamma_c$, then $q$ is a solution of \eqref{eqn:gradSystem} on any interval $I\subset\R$ for which condition {\rm(i)} is strictly satisfied, i.e., $V(q(t))>c$ for all $t\in I$ (see Lemma~\ref{lem:minimisol}). Condition {\rm(ii)} forces $q$ to connect $\mathcal{V}^{c}_-$ to $\mathcal{V}^{c}_+$.  Indeed, if $q$ is a minimizer of $J_{c}$ on $\Gamma_{c}$ there  exists an interval $I=(\alpha,\omega)\subset\R$ (possibly with $\alpha=-\infty$ or $\omega=+\infty$) for which
$V(q(t))>c$ for any $t\in I$ (see Lemma \ref{lem:sol-W-alphabeta}), and 
$$\lim_{t\to\alpha^{+}}\dist(q(t),\mathcal{V}^{c}_-)=\lim_{t\to\omega^{-}}\dist(q(t),\mathcal{V}^{c}_+)=0.$$
Thus, $q$ is a solution of \eqref{eqn:gradSystem} on $I$ and $I$ is a {\sl connecting time interval},  that is to say, an open interval $I\subset\R$ (not necessarily bounded) whose eventual extremes are {\sl contact} times. The existence of a solution to our problem is then obtained by recognizing that the energy of such a minimizer $q$ restricted to $I$ equals $-c$ (see Lemma \ref{lem:energia}), from which we can proceed (by reflection and periodic continuation) to construct our entire connecting solution. Hence, we obtain brake orbit when the connecting interval $I$ is bounded ($\alpha,\omega\in\R$), a homoclinic  when $I$  is an half-line  (precisely one of $\alpha$ and $\omega$ is finite) and finally a heteroclinic if $I$ is the entire real line.\smallskip

In the present paper, we first establish a general existence result for solutions satisfying the aforementioned properties (see section~\S2). In fact, the existence of a minimizer of $J_{c}$ on $\Mm_{c}$ is obtained whenever~\ref{enum:Vc} holds and $J_{c}$ satisfies a mild coerciveness property on $\Gamma_c$, namely,
\begin{equation}\label{eqn:coercivita}
\exists \,R>0 \;\text{ s.t. }\; \inf_{q\in\Gamma_c}{J_{c}}(q)=\inf\{ J_{c}(q): q\in\Gamma_c,\,\|q\|_{L^{\infty}(\R,\R^N)}\leq R\}.\end{equation} 
From this minimizer we can reconstruct a solution $q_{c}\in C^{2}(\R,\R^N)$ to the problem~\eqref{eqn:gradSystem},\eqref{eqn:condbrake} satisfying the energy constraint $E_{q_{c}}(t):=\frac 12|\dot q(t)|^{2}-V(q(t))=-c$ for all $t\in\R$, see Theorem~\ref{P:main1}.\medskip

In order to have a better understanding of the scope of Theorem~\ref{P:main1}, which is presented in a very general form, it might be useful to illustrate some specific situations in which we can verify condition~\ref{enum:Vc} and \eqref{eqn:coercivita}. 
This is done in \S3 where more explicit assumptions on the potential $V$ are considered, including classical cases as double well, Duffing like  and pendulum like potential systems. In all these situations the potential $V$ has isolated minima at the level $c=0$. The 
 application of Theorem~\ref{P:main1} to these cases allows us to obtain existence and multiplicity results of connecting orbits $q_{c}$ at energy level $-c$ whenever $c$ is sufficiently small (see propositions \ref{p:twowell}, \ref{p:duffing} and \ref{p:pendumum}). When $c=0$ the corresponding connecting orbits are homoclinic or heteroclinic solutions connecting the different minima of the potential, while we get brake orbit solutions when $c$ is a regular value for $V$.
We then study  convergence properties  of the family of solutions $q_{c}$ to homoclinic type solutions or heteroclinic type solutions as the energy level $c$ goes to zero (see propositions \ref{p:convergence}, \ref{p:convergenceduffing}  and \ref{p:convergencepend}). \smallskip

 Our results extend recent studies made in \cite{zuniga2018thesis} in the ODE framework, where for a certain class of two-well potentials, periodic orbits of~\eqref{eqn:gradSystem} are shown to converge, in a suitable sense, to a heteroclinic solution joining the wells of such potential.

The issue of existence of heteroclinic solutions connecting the equilibria of multi-well potentials has been quite explored in the literature; the interested reader is referred to~\cite{alikakos2008connection},~\cite{antonopoulos2016minimizers},~\cite{fusco2018existence},~\cite{katzourakis2016ontheloss},~\cite{rabinowitz1993homoclinic}, and~\cite{santambrogio2016metric,sternberg2016heteroclinic} for different approaches on the subject.
\bigskip

{\bf Acknowledgement.} 
This work has been partially supported by a public grant overseen by the French National Research Agency (ANR) as part of the ``Investissements d'Avenir" program (reference: ANR-10-LABX-0098, LabEx SMP), and partially supported by the Projects EFI ANR-17-CE40-0030 (A.Z.) of the French National Research Agency. A.Z. also wishes to thank Peter Sternberg for fruitful discussions on the subject of this paper.


\section{The general existence result}\label{sec:variational}

In this section we state and prove our general result concerning the existence of solutions to the conservative system \eqref{eqn:gradSystem} connecting the sublevels $\mathcal{V}^{c}_{\pm}$ and that satisfy a pointwise energy constraint, provided~\eqref{eqn:coercivita} and \ref{enum:Vc}  hold. The proof of Theorem \ref{P:main1} adapts, to the ODE case, arguments that were already developed in \cite{alessio2013stationary}, \cite{[AlMbrake]} and \cite{alessio2016brake} for (systems of) PDE.

\begin{thm}\label{P:main1}
Assume $V\in C^{1}(\R^{N})$, and that there exists $c\in\R$ such that~\ref{enum:Vc} and the coercivity condition~\eqref{eqn:coercivita} of the energy functional $J_c$ over $\Mm_c$ hold true. Then there exists a solution $q_{c}\in C^{2}(\R,\R^N)$ to~\eqref{eqn:gradSystem}-\eqref{eqn:condbrake} which in addition satisfies 
\[
E_{q_{c}}(t):=\tfrac12|\dot q_{c}(t)|^{2}-V(q_{c}(t))=-c,\;\,\text{ for all }\,t\in\R.
\]
Furthermore, any such solution is classified in one of the following types
{\begin{enumerate}[label={\rm(\alph*)},itemsep=.5em]
\item\label{enum:main:1} $q_{c}$ is of {\sl brake orbit} type: There exist $-\infty<{\sigma<\tau}<+\infty$ so that\\[-0.35cm]
	\begin{enumerate}[label={\rm (\alph{enumi}.\roman*)},itemsep=2pt]
	\item\label{enum:mainthm:a1}  {$V(q_c(\sigma))=V(q_c(\tau))=c$}, $V(q_{c}(t))>c$ for every $t\in (\sigma,\tau)$ and  $ \dot {q_c}(\sigma)= \dot {q_c}(\tau)=0$,\\[-0.35cm]
	\item $q_{c}(\sigma)\in\mathcal{V}^{c}_-$, $q_{c}(\tau)\in\mathcal{V}^{c}_+$, $\nabla V(q_{c}(\sigma)) \neq 0$ and $\nabla V(q_{c}(\tau)) \neq 0$,\\[-0.35cm]
	\item\label{enum:mainthm:a3} $q_{c}(\sigma+t)=q_{c}(\sigma-t)$ and $q_{c}(\tau+t)=q_{c}(\tau-t)$ for all $t\in\R$.
	\end{enumerate}
\item\label{enum:main:2} {\sl $q_{c}$ is of {\sl homoclinic} type:} There exist $\sigma\in\R$ and a component $\mathcal{V}^{c}_{\pm}$ of $\mathcal{V}^{c}$ so that\\[-0.35cm] 
	\begin{enumerate}[label={\rm (\alph{enumi}.\roman*)},itemsep=2pt]
	\item	 $V(q_{c}(\sigma))=c$,  $V(q_{c}(t))>c$ for every $t\in\R\setminus\{\sigma\}$, $ \dot {q_c}(\sigma)=0$ and  $\lim_{t\to\pm\infty}\dot q_{c}(t)= 0$,\\[-0.35cm]
	\item  $q_{c}(\sigma)\in\mathcal{V}^{c}_{\pm}$, $\nabla V(q_{c}(\sigma))\neq 0$ and there exists a closed connected set $\Omega\subset \mathcal{V}^{c}_{\mp}\cap\{ x\in\R^N: V(x)=c,\,\nabla V(x)=0\}$ so  that $\lim_{t\to\pm\infty}\dist(q_{c}(t),\Omega)=0$,\\[-0.35cm]
	\item $q_{c}(\sigma+t)=q_{c}(\sigma-t)$ for all $t\in\R$.\end{enumerate}
\item\label{enum:main:3} {\sl $q_{c}$ is of {\sl heteroclinic} type: There holds} \\[-0.35cm]
	\begin{enumerate}[label={\rm (\alph{enumi}.\roman*)},itemsep=2pt]
	\item $V(q_{c}(t))>c$ for all $t\in\R$ and $\lim_{t\to\pm\infty}\dot q_{c}(t)= 0$,\\[-0.35cm]
 	\item There exist closed connected sets $\mathcal{A} \subset \mathcal{V}^c_-\cap\{ x\in\R^N:V(x)=c,\,\nabla V(x)=0\}$ and $\Omega \subset \mathcal{V}^{c}_+\cap\{ x\in\R^N:V(x)=c,\,\nabla V(x)=0\}$ such that
 \begin{align*}
 \lim\limits_{t\to-\infty}\dist(q_{c}(t),\mathcal{A})=0\;\hbox{ and }\;\lim\limits_{t\to+\infty}\dist(q_{c}(t),\Omega)=0.
 \end{align*}
 	\end{enumerate}
\end{enumerate}}
\end{thm}

{\begin{Remark}\label{rem:regular}
Note that if $c$ is a regular value for $V$ then the corresponding solution $q_c$ given by Theorem~\ref{P:main1} is of brake type, while it may be of the heteroclinic or homoclinic type if $c$ is a critical value of $V$.
\end{Remark}}

\noindent  {To prove Theorem~\ref{P:main1}, }given $c\in\R$ and an interval $I\in\R$, we consider the {\sl action functional} 
\[
J_{c,I}(q):=\int_{I}\tfrac12|\dot q(t)|^{2}+(V(q(t))-c)\, dt,
\]
defined on the space
\[
\Xx_c:=\{q\in H^{1}_{loc}(\R,\R^N):\; \inf_{t\in\R}V(q(t))\geq c\}.
\]
We will write henceforth $J_{c}(q):=J_{c,\R}(q)$.

\begin{Remark}\label{rem:semF}
Note that  since the lower bound $V(q(t))\geq c$ for all $t\in\R$ holds for any $q\in\Xx_c$, then we readily see that $J_{c,I}$ is non-negative on $\Xx_c$ for any given real interval $I$. {Moreover,} $\Xx_c$ is sequentially closed, and for any interval $I\subset\R$,  $J_{c,I}$ is lower semicontinuous with respect to the weak topology of $H^{1}_{loc}(\R,\R^N)$.
\end{Remark}

\begin{Remark}\label{rem:disequazioneF} 
If $q\in\Xx_c$ and $(\sigma,\tau)\subset\R$, then
\begin{align*}
J_{c,(\sigma,\tau)}(q)&\geq\tfrac 1{2(\tau-\sigma)}| q(\tau)-q(\sigma)|^{2}+\int_{\sigma}^{\tau}(V(q(t))-c)\,dt\\
&\geq
\sqrt{\tfrac 2{\tau-\sigma}\int_{\sigma}^{\tau}(V(q(t))-c)\,dt}\ | q(\tau)-q(\sigma)|.
\end{align*}
In particular, if there exists some $\mu>0$ for which $V(q(t))-c\geq \mu\geq 0$ for  all $t\in (\sigma,\tau)$, then we have
\begin{equation}\label{eqn:diseF}
J_{c,(\sigma,\tau)}(q)\geq\sqrt{2\mu}\ | q(\tau)-q(\sigma)|.
\end{equation}
\end{Remark}

\begin{Remark}\label{rem:h}
In view of~\ref{enum:Vc}, the sets $\mathcal{V}^{c}_-$ and $\mathcal{V}^{c}_+$ are disjoint and closed, and so they are locally well separated.
Hence, if $R$ denotes the constant introduced in the coerciveness assumption~\eqref{eqn:coercivita}, we have
\begin{equation}\label{eqn:separazione}
4\rho_{0}:=\dist(\mathcal{V}^{c}_-\cap B_{R}(0), \mathcal{V}^{c}_+\cap B_{R}(0))>0.
\end{equation}

The continuity of $V$ ensures that for any $r>0$ and $C>0$, there exists $h_{r,C}>0$ in such a way that
\[
\inf \{V(x): |x|\leq C\,\hbox{ and }\,\dist(x,\mathcal{V}^{c})\geq r\, \}\geq c+h_{r,C}.
\]
In what follows, we will simply denote $h_{r}:=h_{r,R}$, where $R$ is the constant given in~\eqref{eqn:coercivita}.
\end{Remark}
The variational problem we are interested in studying {involves} the following admissible set
\[
\Mm_c:=\biggl\{u\in\Xx_c:\, \liminf_{t\to-\infty}\dist(q(t),\mathcal{V}^{c}_-)=\liminf_{t\to+\infty}\dist(q(t),\mathcal{V}^{c}_+)=0\biggr\},
\]
{and we will denote $m_c:=\inf_{q\in\Mm_c}J_c(q)$}.
\smallskip

The first observation in place is
\begin{Lemma}\label{lem:m>0}
There results {$m_c\in(0,+\infty)$.}
\end{Lemma}

\begin{proof}[{\bf Proof of Lemma~\ref{lem:m>0}}]
It is plain to observe that $m_c<+\infty$. Indeed, by~\ref{enum:Vc} we can choose two points $\xi_{-}\in\mathcal{V}^{c}_{-}$ and $\xi_{+}\in\mathcal{V}^{c}_{+}$ such that
$V(t\xi_{+}+(1-t)\xi_{-})>c$ for any $t\in (0,1)$. Then defining
\[q(t)=\begin{cases}\xi_{-}& \text{if }\,t\leq 0,\\
t\xi_{+}+(1-t)\xi_{-}& \text{if }\,t\in (0,1),\\
\xi_{+}& \text{if }\,1\leq t.\end{cases}\]
one plainly recognizes that $q\in\G_c$ and $m_c\leq J_c(q)<+\infty$.\\
To show that $m_c>0$, let us observe that in light of~\eqref{eqn:coercivita} 
$$
m_c=\inf\{ J_c(q):\, q\in\Mm_c\;\hbox{ and }\;\|q\|_{L^{\infty}(\R,\R^N)}\leq R\}.
$$
Also, by Sobolev embedding theorems, any $q\in\Mm_c$ is continuous over $\R$ and it verifies 
\[
 \liminf_{t\to-\infty}\dist(q(t),\mathcal{V}^{c}_-)=\liminf_{t\to+\infty}\dist(q(t),\mathcal{V}^{c}_+)=0.
\] 
 Recalling that $\mathcal{V}^{c}=\mathcal{V}^{c}_-\cup\mathcal{V}^{c}_+$, we deduce from~\eqref{eqn:separazione} and the fact that $\|q\|_{L^{\infty}(\R,\R^N)}\leq R$, that there exists a nonempty open interval $(\sigma,\tau)\subset\R$ depending on $q$, in such a way that
 $$|q(\tau)-q(\sigma)|=2\rho_{0},\;\,\hbox{ and }\;\;\dist(q(t),\mathcal{V}^{c})\geq\rho_{0}\; \text{ for any }\,t\in(\sigma,\tau).$$
But then Remark~\ref{rem:h} yields a uniform lower bound $V(q(t))-c\geq h_{\rho_{0}}$, for any $t\in (\sigma,\tau)$. This fact, combined with Remark \ref{rem:disequazioneF} yields 
\[
J_{c,(\sigma,\tau)}(q) \geq \sqrt{2h_{\rho_0}}|q(\tau)-q(\sigma)|\geq \sqrt{2h_{\rho_{0}}}\, 2\rho_{0}.
\]
Therefore, $m_c=\inf_{\Mm_c}J_c(q)\geq \inf_{\Mm_c}J_{c,(\sigma,\tau)}(q)\geq \sqrt{2h_{\rho_{0}}}\, 2\rho_{0}>0$.
\end{proof}

We now argue that finite energy {elements} in the admissible class $\Mm_c$ asymptotically approach the sublevel sets $\mathcal{V}^c_-$ and $\mathcal{V}^{c}_+$ in the following sense
\medskip
\begin{Lemma}\label{lem:limiti}
Suppose $q\in\Mm_c$ satisfies $\|q\|_{L^{\infty}(\R,\R^N)}<+\infty$ and  $J_c(q)<+\infty$. Then 
$$\lim_{t\to-\infty}\dist(q(t),\mathcal{V}^{c}_-)=\lim_{t\to+\infty}\dist(q(t),\mathcal{V}^{c}_+)=0.$$
\end{Lemma}

\begin{proof}[{\bf Proof of Lemma~\ref{lem:limiti}}]
Let us argue the case of the limit as $t\to-\infty$, the other limit can be argued similarly. By definition of $\Mm_c$, $q$ satisfies $\liminf_{t\to-\infty}\dist(q(t),\mathcal{V}^{c}_-)=0$. 
Let us assume by contradiction that $\limsup_{t\to-\infty}\dist(q(t),\mathcal{V}^{c}_-)>0$, so there must exist  $\rho\in (0,\rho_{0})$ and two sequences $\sigma_{n}\to-\infty$, ${\tau_{n}\to-\infty}$ such that $\tau_{n+1}<\sigma_{n}<\tau_{n}$ for which there results $|q(\tau_{n})-q(\sigma_{n})|=\rho$ and $\rho\leq \dist(q(t),\mathcal{V}^{c}_-)\leq2\rho$ for any $t\in(\sigma_{n},\tau_{n})$.
In particular, since $\rho<\rho_{0}$ it follows that $\dist(q(t),\mathcal{V}^{c})>\rho$ for any $t\in(\sigma_{n},\tau_{n})$ and $n\in\N$. Since $M:=\|q\|_{L^{\infty}(\R,\R^N)}<+\infty$, Remark~\ref{rem:h} yields $V(q(t))>c+h_{\rho,M}$ for any $t\in(\sigma_{n},\tau_{n})$ and so, by Remark \ref{rem:disequazioneF}, we conclude
\[
J_{c,(\sigma_{n},\tau_{n})}(q)\geq \sqrt{ 2h_{\rho,M}}\, |q(\tau_{n})-q(\sigma_{n})|=\sqrt{ 2h_{\rho,M}}\,\rho,\,\;\;\text{ for all }\,n\in\N.
\]
But then $J_c(q)\geq\sum_{n=1}^{\infty}J_{c,(\sigma_{n},\tau_{n})}(q)=+\infty$, thus contradicting the assumption $J_c(q)<+\infty$.
\end{proof}
{Moreover, by \eqref{eqn:separazione} we obtain the following concentration result}
\medskip
\begin{Lemma}\label{lem:concentrazione} 
There exists $\bar r\in (0,\frac{\rho_{0}}2)$ so that for any $r\in (0,\bar r)$, there exist $L_{r}>0$, $\nu_{r}>0$, in such a way that for any $q\in\Mm_c$ satisfying: $\dist(q(0),\mathcal{V}^{c})\geq\rho_{0}$, $\| q\|_{L^{\infty}(\R,\R^N)}\leq R$ and $J_c(q)\leq m_c+\nu_{r}$, one has 
\begin{enumerate}[label={\rm(\roman*)},itemsep=2pt]
\item\label{enum:lem:dist1} There is $\tau\in (0,L_{r})$ so that $\,\dist(q(\tau),\mathcal{V}^{c}_+)\leq r$, and $\dist(q(t),\mathcal{V}^{c}_+)< \rho_{0},\,$~for all $t\geq\tau$.
\item\label{enum:lem:dist2} There is $\sigma\in (-L_{r},0)$ so that $\,\dist(q(\sigma),\mathcal{V}^c_-)\leq r$, and $\dist(q(t),\mathcal{V}^c_-)< \rho_{0},\,$~for all $t\leq\sigma$.
\end{enumerate}
\end{Lemma}

\begin{proof}[{\bf Proof of Lemma~\ref{lem:concentrazione}}]
Given any $r\in (0,\frac{\rho_{0}}2)$ we define the following quantities 
\begin{align*}
\mu_{r}&:=\max \{ V(x)-c\,:|x|\leq R\;\hbox{ and }\;\dist(x,\mathcal{V}^{c})\leq r\},\\
L_{r}&:=\tfrac{m_c+1}{h_{r}},\;\;\text{ and }\;\; \nu_{r}:=\tfrac 12 r^{2}+\mu_{r}.
\end{align*}
In view of the continuity of $V$ we have $\lim_{r\to 0^+}\nu_{r}=0$. Hence, we can choose $\bar r\in (0,\frac{\rho_{0}}2)$ so that
\begin{equation}\label{eqn:nuupbdd}
\nu_{r}<\min\left\{\sqrt{ h_{\frac{\rho_{0}}2}}\,\tfrac{\rho_0}4,\, 1\right\}\;\hbox{ for any }\,r\in (0,\bar r).
\end{equation}

\noindent
For $q\in\Mm_c$ satisfying the assumptions of this lemma, let us define 
$$\sigma:=\sup\{ t\in\R:\, \dist(q(t),\mathcal{V}^c_-)\leq r\}, \quad \tau:=\inf\{ t>\sigma:\, \dist(q(t),\mathcal{V}^{c}_+)\leq r\}.$$
We observe that $-\infty<\sigma<\tau<+\infty$, since $\lim_{t\to-\infty}\dist(q(t),\mathcal{V}^c_-)=\lim_{t\to+\infty}\dist(q(t),\mathcal{V}^{c}_+)=0$, in light of the fact that the hypotheses of Lemma~\ref{lem:limiti} are fulfilled for any $q$ as above. Furthermore, the definition of $\sigma$ and $\tau$ yield 
\begin{equation}\label{eq:largedist}
\dist(q(t),\mathcal{V}^{c})>r\;\;\text{ for any }\,t\in (\sigma,\tau).
\end{equation}
Now we claim that
\begin{equation}\label{eqn:smalldistsigma}
\dist(q(t),\mathcal{V}^c_-)<\rho_{0}
\;\;\text{ for any }\,t<\sigma.
\end{equation}

\noindent Indeed, we can fix $\xi_{\sigma}\in\mathcal{V}^c_-$ so that $|q(\sigma)-\xi_{\sigma}|\leq r$ and 
\begin{equation}\label{eq:greaterc}
V((1-s)q(\sigma)+s\xi_{\sigma})>c\;\;\hbox{ for any }\,s\in(0,1).
\end{equation} 
Let us define
\[
\underline{q}(t):=
\left\{\begin{array}{ll}
q(t)& \text{if }\,t\leq\sigma,\\
(\sigma+1-t)q(\sigma)+(t-\sigma)\xi_{\sigma}&\text{if }\,\sigma< t<\sigma+1,\\
\xi_{\sigma}& \text{if }\, \sigma+1\leq t.
\end{array}\right.
\]
and
\[
\overline q(t):=
\left\{\begin{array}{ll}
\xi_{\sigma}& \text{if }\, t\leq\sigma-1,\\
(t-\sigma+1)q(\sigma)+(\sigma-t)\xi_{\sigma}& \text{if }\,\sigma-1<t<\sigma,\\
q(t)& \text{if }\;  \sigma\leq t.
\end{array}\right.
\]
First we note that {by \eqref{eq:greaterc} and since $q\in\Mm_c$, we have} $\overline q\in\Mm_c$ and so $J_c(\overline q)\geq m_c$. The latter, combined with the following inequality
\[
J_{c,(-\infty,\sigma)}(\overline q)\leq\int_{\sigma-1}^{\sigma}\tfrac 12|q(\sigma)-\xi_{\sigma}|^{2}+\mu_{r}\, dt\leq\tfrac 12 r^{2}+\mu_{r}=\nu_{r},
\]
shows that $J_{c,(\sigma,+\infty)}(\overline q)\geq m_c-\nu_r$, from which it follows
\[
m_c+\nu_{r}\geq J_c(q)=J_{c,(-\infty,\sigma)}(\underline{q})+J_{c,(\sigma,+\infty)}(\overline q)\geq J_{c,(-\infty,\sigma)}(\underline{q})+m_c-\nu_{r},
\]
thus proving
\begin{equation}\label{eq:energybddqminus}
J_{c,(-\infty,\sigma)}(\underline{q})\leq 2\nu_{r}.
\end{equation} 
To finish the proof of claim~\eqref{eqn:smalldistsigma}, let us assume by contradiction that there is $t_*<\sigma$ such that $\dist(q(t_*),\mathcal{V}^c_-)\geq\rho_{0}$. Since  $\dist(q(\sigma),\mathcal{V}^c_-)\leq r<\tfrac{\rho_{0}}2$,
we deduce that there exists an interval $(\gamma,\delta)\subset(-\infty,\sigma)$ such that
$|\underline{q}(\delta)-\underline{q}(\gamma)|=\frac{\rho_{0}}2$ and $\tfrac{\rho_{0}}2\leq \dist(\underline{q}(t),\mathcal{V}^{c}_-)\leq \rho_{0}$ for all $t\in(\gamma,\delta)$. Then, estimate~\eqref{eq:energybddqminus} combined with Remark~\ref{rem:disequazioneF} and Remark~\ref{rem:h} (since $\| q\|_{L^{\infty}(\R,\R^N)}\leq R$), yields
\[
2\nu_{r}\geq J_{c,(-\infty,\sigma)}(\underline{q})\geq J_{c,(\gamma,\delta)}(\underline{q})\geq\sqrt{ 2h_{\frac{\rho_{0}}2}}\,\tfrac{\rho_{0}}2,\;\text{ for any }\;r\leq \bar r,
\]
which contradicts~\eqref{eqn:nuupbdd} in view of the definition of $\bar r$. An analogous argument proves that
\begin{equation}\label{eqn:smalldisttau}
\dist(q(t),\mathcal{V}^{c}_+)<\rho_{0}\;\;\text{ for any }\,t>\tau.
\end{equation}
In this way, we have argued that the conditions~\ref{enum:lem:dist1}-\ref{enum:lem:dist2} are satisfied for the choice of $\tau$ and $\sigma$ as above. We are left to prove the chain of inequalities$-L_r<\sigma<0<\tau<L_r$, for the choice of $L_r$ as in the beginning of the proof. To see this, let us first note that $0\in(\sigma,\tau)$. This follows from the way the time $t=0$ was chosen: $\dist(q(0),\mathcal{V}^{c})\geq\rho_{0}$ and from~\eqref{eqn:smalldistsigma}-\eqref{eqn:smalldisttau} combined. Also, from Remark~\ref{rem:h} and~\eqref{eq:largedist} we have $V(q(t))-c\geq h_{r}$ for $t\in (\sigma,\tau)$. This, and~\eqref{eqn:nuupbdd} yield the lower bound
\[
m_c+1> m_c+\nu_{r}\geq J_{c,(0,\tau)}(q)\geq \int^{\tau}_0(V(q(t))-c)dt\geq \tau h_{r}.
\]
In other words, we have proved that $0<\tau<\frac{m_c+1}{h_{r}}=:L_{r}$. Analogously, we derive that $m_c+1> m_c+\nu_{r}\geq J_{c,(\sigma,0)}(q)\geq -\sigma h_{r}$ from which $0<-\sigma<L_{r}$. The proof of Lemma~\ref{lem:concentrazione} is now complete.
\end{proof}
\smallskip

{We can now conclude that the minimal level $m_c$ is achieved in $\Mm_c$. Indeed, we have}

\begin{Lemma}\label{lem:existence} 
There exists $q_{0}\in\Mm_c$ such that $J_c(q_{0})=m_c$.
\end{Lemma}

\begin{proof}[{\bf Proof of Lemma~\ref{lem:existence}}]
Let $(q_{n})\subset\Mm_c$ be a minimizing sequence, so $J_c(q_{n})\to m_c$.
The coerciveness assumption \eqref{eqn:coercivita} allows us to assume that
\begin{equation}\label{eqn:linftybdd}
\|q_{n}\|_{L^{\infty}(\R,\R^N)}\leq R\;\hbox{ for any }\,n\in\N,
\end{equation} 
and hence Lemma \ref{lem:limiti} yields
\[
\lim_{t\to-\infty}\dist(q_{n}(t),\mathcal{V}^c_-)=\lim_{t\to+\infty}\dist(q_{n}(t),\mathcal{V}^{c}_+)=0
\;\;\text{ for any }\,n\in\N.
\]
Thus, \eqref{eqn:separazione} combined with continuity arguments shows that there is $(t_{n})\subset\R$ so that $\dist(q_{n}(t_{n}),\mathcal{V}^{c})=\rho_{0}$. Since the variational problem is invariant under time translations, we can assume that
\begin{equation}\label{eqn:calibra}
\dist(q_{n}(0),\mathcal{V}^{c})=\rho_{0}\;\;\hbox{ for any }\,n\in\N.
\end{equation}
But then conditions~\eqref{eqn:calibra},\eqref{eqn:linftybdd} together with $J_c(q_{n})\to m_c$ allow us to use Lemma~\ref{lem:concentrazione} to deduce the existence of $L>0$, in a such way that
\begin{equation}\label{eqn:conc}
\sup_{t\in(-\infty,-L)}\dist(q_{n}(t),\mathcal{V}^c_-)\leq \rho_{0}\;\;\hbox{ and }\;\,\sup_{t\in(L,+\infty)}\dist(q_{n}(t),\mathcal{V}^{c}_+)\leq\rho_{0},
\end{equation}
for all but finitely many terms in the sequence $(q_n)$. Observe now that $\inf_{t\in\R}V(q_{n}(t))\geq c$ for all $n\in\N$, since $q_{n}\in\Xx_c$, whence
\begin{equation}\label{eqn:qdot}
\|\dot q_{n}\|_{L^{2}(\R,\R^N)}^{2}\leq 2m_c+o(1),\;\;\hbox{ as }\;n\to\infty.
\end{equation}
By \eqref{eqn:linftybdd} and \eqref{eqn:qdot} we obtain the existence of $q_{0}\in H^{1}_{loc}(\R,\R^N)$, such that along a subsequence (which we continue to denote $q_{n}$) $q_{n}\rightharpoonup q_{0}$ weakly in $H^{1}_{loc}(\R,\R^N)$. As $q_{0}\in\Xx_c$, in view of Remark~\ref{rem:semF}, we deduce $J_c(q_{0})\leq m_c=\lim_{n\to\infty}J_c(q_{n})$.
On the other hand, the pointwise convergence, \eqref{eqn:linftybdd} and \eqref{eqn:conc} yield that $V(q_{0}(t))\geq c$ for any $t\in\R$, that $\|q_{0}\|_{L^{\infty}(\R,\R^N)}\leq R$ and 
\begin{equation}\label{eqn:asymptDist}
\sup_{t\in(-\infty,-L)}\dist(q_{0}(t),\mathcal{V}^c_-)\leq \rho_{0},\quad\sup_{t\in(L,+\infty)}\dist(q_{0}(t),\mathcal{V}^{c}_+)\leq\rho_{0}.
\end{equation}
Since $\int_{L}^{+\infty}(V(q_{0}(t))-c)\, dt\leq J_{c,(L,+\infty)}(q_{0})\leq m_c$, we obtain that
\[
\liminf_{t\to+\infty}V(q_{0}(t))-c=0,\;\hbox{ which implies }\;\liminf_{t\to+\infty}\dist(q_{0}(t),\mathcal{V}^{c}_+)=0,
\] 
in view of \eqref{eqn:asymptDist}.
Analogously, we deduce that $\liminf_{t\to-\infty}\dist(q_{0}(t),\mathcal{V}^c_-)=0$.  Thus, we have argued that $q_{0}\in\Mm_c$, which in turn shows the reverse inequality $J_c(q_0)\geq m_c$. The proof of Lemma~\ref{lem:existence} is now complete.
\end{proof}

It will be convenient to introduce the following set of minimizers to the variational problem studied in Lemma~\ref{lem:existence},
\[
\MM_c:=\{q\in\Mm_c:\, \dist(q(0),\mathcal{V}^{c})=\rho_{0},\ J_c(q)=m_c,\ \|q\|_{L^{\infty}(\R,\R^N)}\leq R\}.
\]
The proof of the above lemma reveals that $\MM_c\neq\emptyset$.\medskip

For any $q\in\MM_c$, we introduce the contact times of the trajectory of $q$ with the sublevel sets $\mathcal{V}^{c}_+$ and $\mathcal{V}^c_-$ of the potential by letting
\begin{equation}\label{eqn:alphaq}
\alpha_{q}:=\left\{
\begin{array}{ll}
-\infty & \text{ if }\, q(\R)\cap \mathcal{V}^c_-=\emptyset,\\ 
\sup\{t\in\R: \, q(t)\in \mathcal{V}^c_-\}&\text{ if }\,q(\R)\cap \mathcal{V}^c_-\neq\emptyset.
\end{array}\right.
\end{equation}
and
\begin{equation}\label{eqn:omegaq}
\omega_{q}:=\left\{
\begin{array}{ll}
\inf\{t>\alpha_{q}:\, q(t)\in \mathcal{V}^{c}_+\}&\text{ if }\,q(\R)\cap \mathcal{V}^{c}_+\neq\emptyset,\\
+\infty& \text{ if }\, q(\R)\cap \mathcal{V}^{c}_+=\emptyset. 
\end{array}\right.
\end{equation}
Note that {for all $q\in\MM_c$}, by Lemma \ref{lem:concentrazione} and by the definition of $\rho_0$~\eqref{eqn:separazione}, it is simple to verify that  
$$-\infty\leq \alpha_{q}<\omega_{q}\leq +\infty.$$ 
Moreover, by definition of $\alpha_{q}$ and $\omega_{q}$, {since $\dist(q(0),\mathcal{V}^{c})=\rho_{0}$ for every $q\in\MM_c$},  we have $q(t)\in\R^N\setminus (\mathcal{V}^c_-\cup\mathcal{V}^{c}_+)$ for any $\alpha_q<t<\omega_q$,  {that is} 
\[
V(q(t))>c\;\;\text{ for any }\, t\in (\alpha_{q},\omega_{q}).
\]

{Therefore we obtain}

\smallskip
\begin{Lemma}\label{lem:minimisol} 
If $q\in\MM_c$ then $q\in {C}^{2}((\alpha_{q},\omega_{q}),\R^N)$. Furthermore, any such $q$ is a solution to the system
\[
\ddot q(t)=\nabla V(q(t)),\;\;\text{ for all }\,t\in (\alpha_{q},\omega_{q}).
\]
\end{Lemma}

\begin{proof}[{\bf Proof of Lemma~\ref{lem:minimisol}}]
Let $\psi\in C_{0}^{\infty}(\R)$ be so that $ \hbox{supp}\, \psi\subset [a,b]\subset (\alpha_{q},\omega_{q})$.
Since $V(q(t))>c$ for any $t\in(\alpha_{q},\omega_{q})$ and $t\mapsto V(q(t))$ is continuous on $\R$, we derive that there exists $\lambda_{0}>0$ such that $\min_{t\in [a,b]}V(q(t))=c+\lambda_{0}$. The continuity of $V$ ensures that there exists $h_{\psi}>0$ such that
\[
\min_{t\in [a,b]}V(q(t)+h\psi(t))>c\;\,\hbox{ for any }\;h\in (0,h_{\psi}).
\]
In other words, for any $\psi\in C_{0}^{\infty}(\R)$ with $\hbox{supp}\, \psi\subset [a,b]\subset  (\alpha_{q},\omega_{q})$ there exists $h_{\psi}>0$ in such a way that $q+h\,\psi\in\Mm_c$, provided $h\in (0,h_{\psi})$. Since $q$ is a minimizer of $J_c$ over $\Mm_c$, then
\[
J_c(q+h\psi)-J_c(q)\geq 0\;\,\hbox{ for any }\;h\in (0,h_{\psi}).
\]
Writing the inequality explicitly, and using the Dominated Convergence Theorem as $h\to 0^+$, we readily see 
\begin{align*}
&\phantom{=}\lim_{h\to0^{+}}\frac{1}{h}\left(\int_{\alpha_{q}}^{\omega_{q}}(\tfrac 12|\dot q+h\dot\psi|^{2}+
V(q+h\psi)-c)\,dt-J_{c,(\alpha_{q},\omega_{q})}(q)\right)\\
&=\lim_{h\to 0^{+}}\frac{1}{h}
    \int_{\text{supp }\psi}\left(\tfrac 12(|\dot q+h\dot \psi|^{2}-|\dot q|^{2})+(V(q+h\psi)-V(q))\right)dt\\
&=\int_{\R}\dot q\cdot{\dot\psi}+\nabla V(q)\cdot\psi\, dt\ge0.
\end{align*}
The same argument with $-\psi$ as test function shows $\int_{\R}\dot q\cdot\dot\psi+\nabla V(q)\cdot\psi\, dt=0$, so $q$ is a weak solution of $\ddot q=\nabla V(q)$ on $(\alpha_{q},\omega_{q})$. Standard regularity arguments show that $q\in C^{2}((\alpha_{q},\omega_{q}),\R^N)$, whence $q$ is a strong solution to the above system. 
\end{proof}
{Moreover, we have}
\begin{Lemma}\label{lem:sol-W-alphabeta}
If $q\in\MM_c$, then 
\begin{enumerate}[label={\rm(\roman*)},itemsep=2pt]
	\item\label{enum:lem:disqbdry1} $\lim\limits_{t\to\alpha_{q}^{+}}\dist(q(t),\mathcal{V}^c_-)=0$, and if $\alpha_{q}>-\infty$ then $V(q(\alpha_{q}))=c$ with $q(\alpha_{q})\in\mathcal{V}^c_-$,	
	 \item\label{enum:lem:disqbdry2} $\lim\limits_{t\to\omega_{q}^{-}}\dist(q(t),\mathcal{V}^{c}_+)=0$, and if $\omega_{q}<+\infty$ then $V(q(\omega_{q}))= c$ with $q(\omega_{q})\in\mathcal{V}^{c}_+$.
\end{enumerate}
\end{Lemma}

\begin{proof}[{\bf Proof of Lemma~\ref{lem:sol-W-alphabeta}}]
  If $\alpha_{q}=-\infty$, by Lemma \ref{lem:limiti} we obtain $\lim_{t\to -\infty}\dist(q(t),\mathcal{V}^c_-)=0$.  
  If  $\alpha_{q}>-\infty$, then the continuity of $q$ and the definition of $\alpha_{q}$ imply that $q(\alpha_{q})\in\mathcal{V}^c_-$, and $V(q(\alpha_{q}))= c$. In particular, $\lim_{t\to \alpha_{q}^{+}}V(q(t))=c$, and~\ref{enum:lem:disqbdry1} follows. One argues~\ref{enum:lem:disqbdry2} in a similar fashion.
 \end{proof}

{By the previuos result we obtain }

\begin{Lemma}\label{lem:sol-minimoalphaomega}
 Any $q\in\MM_c$ satisfies $J_c(q)=J_{c,(\alpha_{q},\omega_{q})}(q)=m_c$. Moreover, $q(t)\equiv q(\alpha_{q})$ on $(-\infty,\alpha_{q})$ if $\alpha_{q}\in\R$, and $q(t)\equiv q(\omega_{q})$ on $(\omega_{q},+\infty)$ if $\omega_{q}\in\R$.
 \end{Lemma}
 
 \begin{proof}[{\bf Proof of Lemma~\ref{lem:sol-minimoalphaomega}}] 
 Let us define $\tilde q$ to be equal to $q$ on the interval $(\alpha_{q},\omega_{q})$, and such that
 $\tilde q(t)=q(\alpha_{q})$ on $(-\infty,\alpha_{q})$ if $\alpha_{q}\in\R$,  while $\tilde q(t)=q(\omega_{q})$ on $( \omega_{q},+\infty)$ if $\omega_{q}\in\R$ (so if neither of $\alpha_q$ or $\omega_q$ is finite, then $\tilde{q}=q$). In view of Lemma \ref{lem:sol-W-alphabeta} we see that $\tilde q\in\Mm_c$, whence $J_c(\tilde q)\ge m_c$. The latter implies that $\tilde{q}$ is also a minimizer of $J_c$, as $J_c(\tilde q)=J_{c,(\alpha_{q},\omega_{q})}(\tilde q)=J_{c,(\alpha_{q},\omega_{q})}(q)\leq m_c$, from which we deduce $m_c=J_c(\tilde q)=J_{c,(\alpha_{q},\omega_{q})}(q)$.  In particular,  since $\inf_{t\in\R}V(q(t))\geq c$ we obtain $J_{c,(-\infty, \alpha_{q})}(q)=J_{c,(\omega_{q},+\infty)}(q)=0$, which shows $\|\dot q\|_{L^{2}((-\infty, \alpha_{q}),\R^N)}=\|\dot q\|_{L^{2}((\omega_{q},+\infty),\R^N)}=0$. Therefore, $q$ must be constant on $(-\infty, \alpha_{q})$, and on $(\omega_{q},+\infty)$, so the lemma is established.
 \end{proof}
 
{By the previous result, we obtain}

\smallskip
\begin{Lemma}\label{lem:sol-pohoz}
Consider $q\in\MM_c$, and let $(\tau,\sigma)\subseteq (\alpha_{q},\omega_{q})$ be arbitrary. Then,
\[
\tfrac12\int_{\tau}^{\sigma}|\dot q(t)|^{2}\, dt=\int_{\tau}^{\sigma} (V(q(t))-c)\,dt.
\]
\end{Lemma}
\begin{proof}[{\bf Proof of Lemma~\ref{lem:sol-pohoz}}] 
Given any $\tau\in (\alpha_{q},\omega_{q})$, we will prove that
\begin{equation}\label{eqn:tau}
\tfrac12\int_{\tau}^{\omega_{q}}|\dot q(t)|^{2}\, dt=\int_{\tau}^{\omega_{q}} (V(q(t))-c)\, dt.
\end{equation}
For any $s>0$ and $\tau$ as above, we define
\[
q_{s}(t):=\left\{
\begin{array}{ll}
q(t+\tau)& \text{if }\;t\leq 0,\\[.25em]
q(\frac ts+\tau)&\text{if }\; t>0.
\end{array}\right.
\]
It is easy to check that $q_{s}\in\Mm_c$, by Lemma \ref{lem:sol-W-alphabeta} using that $q\in\MM_c$. Furthermore, by Lemma \ref{lem:sol-minimoalphaomega}, we deduce
\begin{align*}
J_c(q_{s})=J_{c,(\alpha_{q}-\tau, s(\omega_{q}-\tau)))}(q_{s})&\geq m_c=J_{c,(\alpha_{q}, \omega_{q})}(q)
=J_{c,(\alpha_{q}-\tau, \omega_{q}-\tau)}(q(\cdot+\tau)).
\end{align*}
In particular, {for all $s>0$} we obtain the following
\begin{align*}
0&\leq J_{c,(\alpha_{q}-\tau, s(\omega_{q}-\tau))}(q_{s})-J_{c,(\alpha_{q}-\tau, \omega_{q}-\tau)}(q(\cdot+\tau))\\[0.1cm]
&= \int_{0}^{s(\omega_{q}-\tau)}\bigl(\tfrac 12|\dot q_{s}(t)|^{2}+V(q_{s}(t))-c\bigr)\,dt-\int_{0}^{\omega_{q}-\tau}\bigl(\tfrac 12|\dot q(t+\tau)|^{2}+V(q(t+\tau))-c\bigr)\, dt\\
&= \int_{0}^{s(\omega_{q}-\tau)}\left(\tfrac{1}{2s^{2}}\bigl|\dot q\bigl(\tfrac{t}{s}+\tau\bigr)\bigr|^{2}+V\bigl(q\bigl(\tfrac{t}{s}+\tau\bigr)\bigr)-c\right) dt-J_{c,(\tau, \omega_{q})}(q)\displaybreak[3]\\
&= \frac{1}{s}\int_{\tau}^{\omega_{q}}\tfrac 12|\dot q(t)|^{2}\, dt+s\int_{\tau}^{\omega_{q}}(V(q(t))-c)\, dt-J_{c,(\tau, \omega_{q})}(q)\\
&= \bigl(\tfrac{1}{s}-1\bigr)\int_{\tau}^{\omega_{q}}\tfrac 12|\dot q(t)|^{2}\, dt
+(s-1)\int_{\tau}^{\omega_{q}}(V(q(t))-c)\, dt.
\end{align*}
{Hence, setting } $T:=\int_{\tau}^{\omega_{q}}\tfrac12|\dot q(t)|^{2}\, dt$ and $U:=\int_{\tau}^{\omega_{q}}(V(q(t))-c)\, dt$, {we get  that}  the real function $s\mapsto f(s)=(\tfrac 1s-1)T+(s-1)U$ is non-negative over $(0,+\infty)$. {Since } it achieves a non-negative minimum  at $s=\sqrt{T/U}$, where
$$
0\leq\; f\biggl(\sqrt{\tfrac{T}{U}}\biggr)=\left(\sqrt{\tfrac{U}{T}}-1\right)T+\left(\sqrt{\tfrac{T}{U}}-1\right)U=-\biggl(\sqrt{T}-\sqrt{U}\biggr)^{2},
$$
{we conclude}  $T=U$, i.e. \eqref{eqn:tau}.\hfill\break
A similar argument also shows that for any $\sigma\in(\alpha_{q},\omega_{q})$ one has the identity
\begin{equation}\label{eqn:sigma}
\tfrac12\int_{\alpha_{q}}^{\sigma}|\dot q(t)|^{2}\, dt=\int_{\alpha_{q}}^{\sigma} (V(q(t))-c)\, dt.
\end{equation}
{Then,  by \eqref{eqn:tau} and \eqref{eqn:sigma}, using the additivity property of the integral, we conclude the proof.} 
\end{proof}

{We are now able to prove that every $q\in\MM_c$ satisfies the following pointwise energy constraint:}

\begin{Lemma}\label{lem:energia}
Every $q\in\MM_c$ verifies
\[
E_{q}(t):=\tfrac12|\dot q(t)|^{2}-V(q(t))= -c,\;\,\text{ for all }\;t\in (\alpha_{q},\omega_{q}).
\]
\end{Lemma}

\begin{proof}[{\bf Proof of Lemma~\ref{lem:energia}}] 
 By Lemma~\ref{lem:minimisol}, $q$ solves the system of differential equations~\eqref{eqn:gradSystem} on $(\alpha_{q},\omega_{q})$ and so the energy $E_{q}(t)=\frac 12|\dot q(t)|^{2}-V(q(t))$ must be constant on $(\alpha_{q},\omega_{q})$. We are left to show that the value of this constant is precisely $-c$. Let us first treat the case $\alpha_{q}=-\infty$. We observe that $J_c(q)=\int_\R\frac 12|\dot q(t)|^{2}+(V(q(t))-c)=m_c<+\infty$ direcly yields 
 \[
 \liminf_{t\to -\infty}\tfrac12|\dot q(t)|^{2}+V(q(t))-c=0.
 \]
Since $V(q(t))\geq c$ for any $t\in\R$, we deduce that there exists a sequence $(t_{n})$ such that $t_{n}\to -\infty$  for which $\lim_{n\to+ \infty}\tfrac12|\dot q(t_{n})|^{2}=0$ and $\lim_{n\to\infty}V(q(t_{n}))=c$. So necessarily
 \[
 \lim_{n\to+\infty}\tfrac12|\dot q(t_{n})|^{2}-V(q(t_{n}))=-c,
 \]
{Hence $ E_{q}(t)=-c$ for all $t\in(-\infty,\omega_q)$, proving }  the lemma in the case $\alpha_{q}=-\infty$. Clearly, the argument above can be easily applied when $\omega_{q}=+\infty$, to show $E_{q}(t)= -c$ {for all $t\in(\alpha_q,+\infty)$}.\\
Let us consider now the case $-\infty<\alpha_{q}<\omega_{q}<+\infty$.  As $q\in\MM_c$, Lemma \ref{lem:sol-W-alphabeta} tells us that $V(q(\omega_{q}))= c$, and so by continuity of the potential it follows ${\lim_{t\to\omega_{q}^{-}}V(q(t))=c}$. A similar continuity argument shows
\[
\lim_{y\to\omega_{q}^{-}}\frac 1{\omega_{q}-y}\int_{y}^{\omega_{q}}(V(q({t}))-c)\,dt=0,
\]
which, in light of the identity of Lemma \ref{lem:sol-pohoz}, directly proves that
\[
\lim_{y\to\omega_{q}^{-}}\frac 1{\omega_{q}-y}\int_{y}^{\omega_{q}}\tfrac12|\dot q(t)|^{2}\,dt=0.
\]
For then, $\liminf_{y\to\omega_{q}^{-}}\tfrac12|\dot q(t)|^{2}=0$, from which
\[
 \liminf_{t\to \omega_{q}^{-}}\tfrac12|\dot q(t)|^{2}-V(q(t))=\liminf_{y\to\omega_{q}^{-}}\tfrac12|\dot q(t)|^{2}-\lim_{t\to \omega_{q}^{-}}V(q(t))=-c,
\]
{proving that $E_q(t)=-c$ for every $t\in(\alpha_q,\omega_q)$.}
\end{proof}

{We are now able to construct the connecting solutions, concluding the Proof of Theorem~\ref{P:main1}}

\medskip
{For $q\in\MM_c$  and provided $\omega_{q}<+\infty$, we denote $q_{+}$ the {\sl extension of $q$ by reflection} with respect to $\omega_{q}$\medskip
\begin{align}
q_{+}(t)& :=\left\{
\begin{array}{ll}
q(t)& \text{if }\,t\in (\alpha_{q},\omega_{q}],\\ 
q(2\omega_{q}-t)& \text{if }\,t\in (\omega_{q},2\omega_{q}-\alpha_{q}).
\end{array}\right.\label{eqn:extensions}
\end{align}
Similarly, for $q\in\MM_c$ with $\alpha_q>-\infty$, let $q_{-}$ be the {\sl extension of $q$ by reflection} with respect to  $\alpha_{q}$
\begin{align}
 q_{-}(t)&:=\left\{
\begin{array}{ll}
q(2\alpha_{q}-t)& \text{if }\,t\in (2\alpha_{q}-\omega_{q},\alpha_{q}),\\ 
q(t)&\text{if }\,t\in [\alpha_{q},\omega_{q}).
\end{array}\right.\notag
\end{align}}

We have
\begin{Lemma}\label{lem:brake}
For $q\in\MM_c$, the following properties hold:
\begin{enumerate}[label=$\bullet$,itemsep=2pt]
\item If $\omega_{q}<+\infty$, then $\lim_{t\to \omega_{q}^{-}}\dot q_+(t)=0$, and $q_{+}$ 
is a solution of \eqref{eqn:gradSystem} on $(\alpha_{q},2\omega_{q}-\alpha_{q})$. Furthermore, there results $\nabla V(q_{+}(\omega_{q}))\neq 0$.
\item If $\alpha_{q}>-\infty$, then $\lim_{t\to \alpha_{q}^{+}}\dot q_-(t)=0$, and $q_{-}$ is a solution of \eqref{eqn:gradSystem} on $(2\alpha_{q}-\omega_{q},\omega_{q})$. Furthermore, there results $\nabla V(q_{-}(\alpha_{q}))\neq 0$.
\end{enumerate}  
 \end{Lemma}

\begin{proof}[{\bf Proof of Lemma~\ref{lem:brake}}] 
Given $q\in\MM_c$, let us assume $\omega_{q}<+\infty$; the other case where $\alpha_{q}>-\infty$ can be treated analogously.  By Lemma \ref{lem:sol-W-alphabeta} we already know that $V(q(\omega_{q}))= c$, so by continuity, $\lim_{t\to\omega_{q}^{-}}V(q(t))-c=0$, which in turn shows $\lim_{t\to\omega_{q}^{-}}|\dot q(t)|^{2}=\lim_{t\to\omega_{q}^{-}}2(V(q(t)-c)=0$, due to Lemma \ref{lem:energia}.\\
The system~\eqref{eqn:gradSystem} is of second order and autonomous, so starting from the solution $q$ of \eqref{eqn:gradSystem} over $(\alpha_{q},\omega_{q})$ (in view of~Lemma~\ref{lem:minimisol}) we immediately get that $q_{+}$ is a solution of \eqref{eqn:gradSystem} on $(\alpha_{q},\omega_{q})\cup$ ${(\omega_{q},2\omega_{q}-\alpha_{q})}$. Since $q_{+}$ is continuous on the entire interval $(\alpha_{q}, 2\omega_{q}-\alpha_{q})$ and $\lim_{t\to\omega_{q}}\dot q_{+}(t)=0$ as argued in the preceding paragraph, we deduce that $q_{+}\in C^{1}(\alpha_{q}, 2\omega_{q}-\alpha_{q})$. Using now the fact that $q_{+}$ solves \eqref{eqn:gradSystem} on $(\alpha_{q},2\omega_{q}-\alpha_{q})\setminus\{\omega_q\}$, we readily see that the second derivative exists $\ddot q_{+}(\omega_{q}):=\lim_{t\to\omega_{q}}\ddot q_{+}(t)=\nabla V(q_{+}(\omega_{q}))$. For then $q_{+}\in C^{2}(\alpha_{q}, 2\omega_{q}-\alpha_{q})$, and furthermore, it solves \eqref{eqn:gradSystem} on the entire interval $(\alpha_{q}, 2\omega_{q}-\alpha_{q})$.\par\noindent
In order to conclude the proof, we need to argue that $\nabla V(q_{+}(\omega_{q}))\neq 0$. Suppose on the contrary that $\nabla V(q_{+}(\omega_{q}))=0$, then $q(\omega_{q})$ is an equilibrium of \eqref{eqn:gradSystem}. Since $\dot q_{+}(\omega_{q})=0$ and $q_{+}(\omega_{q})=q(\omega_{q})$,
the uniqueness of solutions to the Cauchy problem shows $q_{+}(t)= q(\omega_{q})$ for any $t\in (\alpha_{q}, 2\omega_{q}-\alpha_{q})$. 

However, this contradicts Lemma \ref{lem:sol-W-alphabeta}, for which $q_+(\omega_q)=q(\omega_{q})\in\mathcal{V}^{c}_+$ and 
\[
\lim_{t\to\alpha_{q}^{+}}\dist(q_+(t),\mathcal{V}^c_-)=\lim_{t\to\alpha_{q}^{+}}\dist(q(t),\mathcal{V}^c_-)=0.
\]
\end{proof}
Thanks to Lemma \ref{lem:brake}, we know that in the event $q\in\MM_c$ satisfies $-\infty<\alpha_{q}<\omega_{q}<+\infty$, then $q_{+}$ is a solution on the bounded interval $(\alpha_{q},2\omega_{q}-\alpha_{q})$, and it verifies
\begin{align*}
&\lim_{t\to\alpha_{q}^{+}}\dot q_{+}(t)=\lim_{t\to(2\omega_{q}-\alpha_{q})^{-}}\dot q_{+}(t)=0\;\;\hbox{ and }\;\;
\lim_{t\to\alpha_{q}^{+}}q_{+}(t)=\lim_{t\to(2\omega_{q}-\alpha_{q})^{-}}q_{+}(t)=q(\alpha_{q}).
\end{align*}

This property implies, in particular, that the $2(\omega_{q}-\alpha_{q})$-periodic extension of $q_{+}$ is well defined. In fact, by Lemma \ref{lem:minimisol} this extension is a classical $2(\omega_{q}-\alpha_{q})$-periodic solution of \eqref{eqn:gradSystem}. Clearly, one also makes analogous statements for $q_{-}$.\smallskip

 {Hence, in the case  $q\in\MM_c$ is so that $-\infty<\alpha_{q}<\omega_{q}<+\infty$, we denote $T:=2(\omega_{q}-\alpha_{q})$, { and we let $q_c$ be the {\sl T-periodic extension} of $q_{+}$ {\em (or $q_{-}$)} over $\R$, obtaining that  $q_c$ is a $T$-periodic classical solution of \eqref{eqn:gradSystem} over $\R$, that satisfies the pointwise energy constraint $E_{{q_c}}(t)=-c$ for all $t\in\R$.  Furthermore, by Lemma \ref{lem:sol-W-alphabeta} and Lemma \ref{lem:brake}, it connects $\mathcal{V}^c_-$ to $\mathcal{V}^{c}_+$, in the following sense
\begin{enumerate}[label={\rm(\roman*)},itemsep=2pt]
 \item[(i)] $V(q_c(\alpha_{q}))=V(q_c(\omega_{q}))=c$, $V(q_c(t))>c$ for any $\,t\in (\alpha_{q},\omega_{q})$ and $ \dot {q_c}(\alpha_{q})= \dot {q_c}(\omega_{q})=0$,
\item[(ii)] $q_c(\alpha_{q})\in\mathcal{V}^{c}_-$, $q_c(\omega_{q})\in\mathcal{V}^{c}_+$, $\nabla V(q_{c}(\alpha_q)) \neq 0$ and $\nabla V(q_{c}(\omega_q)) \neq 0$,
\item[(iii)] $q_c(\alpha_{q}-t)=q_c(\alpha_{q}+t)\,$ and $\;q_c(\omega_{q}-t)=q_c(\omega_{q}+t)$, for any $t\in \R$.
\end{enumerate}}
 \medskip

{Therefore,} when a minimizer $q\in\MM_c$ satisfies $-\infty<\alpha_{q}<\omega_{q}<+\infty$, then it generates a {solution which periodically oscillates back and forth between the boundary of the two sets $\mathcal{V}^{c}_-$ and $\mathcal{V}^{c}_+$, that is a }{\sl brake orbit type} solution of \eqref{eqn:gradSystem} connecting $\mathcal{V}^{c}_-$ and $\mathcal{V}^{c}_+$. {Moreover, it is bounded, and in fact, verifies \eqref{eqn:condbrake}}.  {Hence, denoting $\sigma=\alpha_q$ and $\tau=\omega_q$, the assertion (a) in Theorem~\ref{P:main1} is proved.}
\smallskip

The remaining cases, where {the} minimizer $q\in\MM_c$ satisfies either $\omega_{q}=+\infty$, or $\alpha_{q}=-\infty$, are dealt in Lemma~\ref{lem:omegalimite} below. Before stating this lemma, let us introduce some notation. For $q\in\MM_c$ having $\omega_{q}=+\infty$, we will denote the $\omega$-limit set of $q$ by
\[
\Omega_{q}:=\bigcap_{t>\alpha_{q}}\overline{\{q(s):s\geq t\}}.
\]
The boundedness of any minimizer $q\in\MM_c$, $\|q\|_{L^{\infty}(\R,\R^N)}\leq R$, shows that $\Omega_{q}\subset \overline{B_{R}(0)}$. Also  $\Omega_{q}$ is a closed, connected subset of $\R^n$, being the intersection of closed connected sets. Analogously, for $q\in\MM_c$ having $\alpha_{q}=-\infty$, we will write 
\[
\mathcal{A}_{q}:=\bigcap_{t<\omega_{q}}\overline{\{ q(s): s\leq t\}},
\]
for the $\alpha$-limit  set of $q$, which is a closed, connected subset of $\overline{B_{R}(0)}$. Hence, we have

\begin{Lemma}\label{lem:omegalimite}
 Suppose $q\in\MM_c$ has either $\alpha_{q}=-\infty$, or $\,\omega_{q}=+\infty$. Then, 
 $\dot q(t)\to 0$ as $t\to-\infty$, or as $t\to +\infty$, respectively. Moreover the $\alpha$-limit set of $q$, or the $\omega$-limit set of $q$, respectively, is constituted by critical points of $V$ at level $c$,  {namely}
 \begin{enumerate}[label=$\bullet$,itemsep=2pt]
 \item $\mathcal{A}_{q}\subset \mathcal{V}^c_-\cap\{\xi\in\R^N: V(\xi)=c,\,\nabla V(\xi)=0\}\,$,~or respectively,
 \item $\Omega_{q}\subset \mathcal{V}^{c}_+\cap\{\xi\in\R^N: V(\xi)=c,\,\nabla V(\xi)=0\}$. \end{enumerate}
 \end{Lemma}
 
\begin{proof}[{\bf Proof of Lemma~\ref{lem:omegalimite}}] 
Given $q\in\MM_c$, let us assume $\omega_{q}=+\infty$; the case $\alpha_{q}=-\infty$ is treated similarly. 
{First, note that $\lim_{t\to +\infty}\dot q(t)=0$. Indeed, the fact that $\lim_{t\to +\infty}\dist(q(t),\mathcal{V}^{c}_+)=0$ (cf.~Lemma \ref{lem:limiti}) combined with the uniform continuity of $V$ on $\overline{B_{R}(0)}$, where $\|q\|_{L^{\infty}(\R,\R^N)}\leq R$, proves that $\lim_{t\to +\infty}V(q(t))=c$. But this together with Lemma \ref{lem:energia} show that $\lim_{t\to +\infty}|\dot q(t)|^{2}=\lim_{t\to +\infty}2(V(q(t))-c)=0$.}\\
To prove the second statement, let $\xi\in \Omega_{q}$ so there is a sequence $t_{n}\to +\infty$ such that $q(t_{n})\to \xi$ as $n\to+\infty$. Since  $\lim_{t\to +\infty}\dist(q(t),\mathcal{V}^{c}_+)=0$, it follows $\dist(\xi,\mathcal{V}^{c}_+)=$ $\lim_{n\to\infty}\dist(q(t_{n}),\mathcal{V}^{c}_+)=0$, whereby $\xi\in \mathcal{V}^{c}_+$ in light of the closeness of $\mathcal{V}^{c}_+$.  On the other hand, we have already seen that $\lim_{t\to +\infty}V(q(t))=c$, so the continuity of the potential yields $V(\xi)=\lim_{n\to\infty}V(q(t_{n}))=c$. All of this shows that $\Omega_{q}\subset \mathcal{V}^{c}_+\cap\{\xi\in\R^N:V(\xi)=c\}$. There just remains to be shown that $\nabla V(\xi)=0$. \\
For $(t_n)\subset\R$ given as above, consider the following sequence of translates of $q$, 
\[
q_{n}(t)=q(t+t_{n})\;\,\text{ for }\;t\in (\alpha_{q}-t_{n},+\infty)\;\text{ and }\;n\in\N,
\]
 and note that for any bounded interval $I\subset\R$ we have that $\sup_{t\in I}|\dot q_{n}(t)|\to 0$ as $n\to+\infty$, because $\dot q(t)\to 0$ as $t\to+\infty$. In particular, for any $t\in\R$ we deduce that 
\[
\lim_{n\to+\infty}q_{n}(t)=\lim_{n\to+\infty}q_{n}(0)+\lim_{n\to+\infty}\int_{0}^{t}\dot q_{n}(s)\,ds=\lim_{n\to+\infty}q(t_n)=\xi.
\]
Put another way, $q(\cdot+t_{n})\to\xi$ as $n\to+\infty$, with respect to the $C^{1}_{loc}(\R,\R^N)$-topology. Since $\ddot q(t+t_{n})=\nabla V(q(t+t_{n}))$ for $t\in (\alpha_{q}-t_{n},+\infty)$, we conclude that $\ddot q_{n}(t)\to\nabla V(\xi)$ as $n\to+\infty$, uniformly on bounded subsets of $\R$. On the other hand, given any $t\neq 0$, we can take the limit as $n\to+\infty$ in the identity 
\[
\dot q_{n}(t)-\dot q_{n}(0)=\int_{0}^{t}\ddot q_{n}(s)\, ds.
\]
As argued before, the left side of the equation converges to $0$, while the right side converges to $t\nabla V(\xi)$. Therefore, $\nabla V(\xi)=0$, which concludes the proof.
\end{proof}

{We remark that the previous result proves, in particular, that if $c$ is a regular value for $V$ then for every $q\in\MM_c$ both $\alpha_q$ and $\omega_q$ must be finite.}\medskip

{For $q\in\MM_c$ with $\alpha_{q}=-\infty$ or $\omega_{q}=+\infty$, Lemma \ref{lem:brake} and Lemma \ref{lem:omegalimite} allow us to construct from it an entire solution of \eqref{eqn:gradSystem} with energy at level $-c$, connecting $\mathcal{V}^c_-$ and $\mathcal{V}^c_+$ in the sense of \eqref{eqn:condbrake}.
This entire solution is either a {\sl homoclinic type} solution or a {\sl heteroclinic type} solution, depending on the finiteness of $\alpha_{q}$ and $\omega_{q}$. Indeed, let us define }
\[
q_{c}(t):=
\left\{\begin{array}{ll}
q_{+}&\text{ if }\,-\infty=\alpha_{q}\;\text{ and }\;\omega_{q}<+\infty,\\[0.1cm]
q(t)&\text{ if }\,-\infty=\alpha_{q}\;\text{ and }\;\omega_{q}=+\infty,\\[0.1cm]
q_{-}&\text{ if }\,-\infty<\alpha_{q}\;\text{ and }\;\omega_{q}=+\infty.
\end{array}\right.
\]
and observe that in light of Lemma~\ref{lem:energia} and Lemma~\ref{lem:brake}, every subcase in the definition of $q_{c}$ solves \eqref{eqn:gradSystem} and has energy $E_{q_{c}}(t)=-c$ for all $t\in\R$. The way the remaining condition \eqref{eqn:condbrake} is fulfilled, depends on whether the trajectory of $q$ has finite contact times with $\mathcal{V}^c_-$ and $\mathcal{V}^{c}_+$, or if it accumulates at infinity, as in \eqref{eqn:accatinfty}.\smallskip

In the case where $q\in\MM_c$ has precisely one of $\alpha_{q}$ and $\omega_{q}$ finite, we say that $q_{c}$ is a {\sl homoclinic type} {solution} connecting $\mathcal{V}^c_-$ and $\mathcal{V}^{c}_+$. This {solution} satisfies the following properties
\begin{enumerate}[label=$\bullet$, itemsep=2pt]
\item If $\alpha_{q}=-\infty$ and $\omega_{q}<+\infty$, then we have $q_{c}:=q_{+}$. 
 In particular, adopting the notation of Theorem~\ref{P:main1} we denote $\Omega:=\mathcal{A}_{q}$ (=$\mathcal{A}_{q_{c}}=\Omega_{q_{c}}$) and $\sigma:=\omega_{q}$, thus $\Omega$ is a closed connected set and $\sigma<+\infty$. From Lemma \ref{lem:sol-W-alphabeta}, Lemma  \ref{lem:brake} and Lemma \ref{lem:omegalimite}, the definition of $q_{+}$ in \eqref{eqn:extensions} together with $\sigma=\omega_{q}$, it follows that 
{ \begin{itemize}
 \item[(i)] $V(q_{c}(\sigma))=V(q(\omega_{q}))=c$, $V(q_{c}(t))>c$ for any $t\neq\sigma$ and $\lim_{t\to\pm\infty}\dot{q_{c}}(t)=0$,
 \item[(ii)]  $q_{c}(\sigma)\in\mathcal{V}^{c}_+$ and $\nabla V(q_{c}(\sigma))=0$. Moreover, $\Omega\subset {\mathcal{V}^c_-\cap\{ \xi\in\R^N: {V(\xi)=c},\,\nabla V(\xi)=0\}}$, and $\lim_{t\to\pm\infty}\dist(q_{c}(t),\Omega)=0$,
 \item[(iii)] $q_{c}(\sigma+t)=q_{c}(\sigma-t)$ for any $t\in\R$.
 \end{itemize}}

\noindent Finally we remark that~\eqref{eqn:condbrake} is satisfied. Indeed, {by Lemma \ref{lem:sol-W-alphabeta}}, we get that the infimum $\inf_{t\in\R}\dist(q_{c}(t),\mathcal{V}^{c}_+)=0$ is achieved at $t=\sigma$ and ${\lim_{t\to\pm\infty}\dist(q_{c}(t),\mathcal{V}^c_-)=0}$, so in particular, $\inf_{t\in\R}\dist(q_{c}(t),\mathcal{V}^c_-)=0$. 

\item A similar reasoning allows us to conclude that, when $\alpha_{q}>-\infty$ and $\omega_{q}=+\infty$, the function $q_{c}:=q_{-}$ is a {\sl homoclinic type} solution connecting $\mathcal{V}^c_-$ and $\mathcal{V}^{c}_+$, for which we have that $\Omega:=\Omega_q$ ($=\Omega_{q_{c}}=\mathcal{A}_{q_{c}}$) is a connected closed set and setting $\sigma:=\alpha_{q}$, we get
{ \begin{itemize}
 \item[(i)] $V(q_{c}(\sigma))=V(q(\alpha_{q}))=c$, $V(q_{c}(t))>c$ for any $t\neq\sigma$ and $\lim_{t\to\pm\infty}\dot{q_{c}}(t)=0$,
 \item[(ii)]  $q_{c}(\sigma)\in\mathcal{V}^{c}_-$ and $\nabla V(q_{c}(\sigma))=0$. Moreover, $\Omega\subset {\mathcal{V}^c_+\cap\{ \xi\in\R^N: {V(\xi)=c},\,\nabla V(\xi)=0\}}$, and $\lim_{t\to\pm\infty}\dist(q_{c}(t),\Omega)=0$,
 \item[(iii)] $q_{c}(\sigma+t)=q_{c}(\sigma-t)$ for any $t\in\R$.
 \end{itemize}}
{Moreover, also in this case~\eqref{eqn:condbrake} is satisfied}.
\par\noindent
\end{enumerate}

In the remaining case, where $q\in\MM_c$ has $\alpha_{q}=-\infty$ and $\omega_{q}=+\infty$, we say that $q_{c}$ is a {\sl heteroclinic type} {solution} connecting $\mathcal{V}^c_-$ and $\mathcal{V}^{c}_+$. 
\begin{itemize}
\item Clearly $q_{c}=q$. Adopting the notation of Theorem~\ref{P:main1}, we will write $\mathcal{A}:=\mathcal{A}_{q}\, (=\mathcal{A}_{q_c})$ and $\Omega:=\Omega_{q}\,(=\Omega_{q_c})$, whence $\mathcal{A}$ and $\Omega$ are closed connected sets.  {Then, since $\alpha_{q}=-\infty$ and $\omega_{q}=+\infty$, by definition and Lemma~\ref{lem:omegalimite}, we obtain

 \begin{itemize}
 \item[(i)] $V(q_{c}(t))>c$ holds for any $t\in\R$ and $\lim_{t\to\pm\infty}\dot q_{c}(t)= 0$,
 \item[(ii)] $\mathcal{A}\subset\mathcal{V}^c_-\cap\{ \xi\in\R^N:V(\xi)=c,\,\nabla V(\xi)=0\}$, $\Omega\subset\mathcal{V}^{c}_+\cap\{ \xi\in\R^N:V(\xi)=c,\,\nabla V(\xi)=0\}$ and $\lim_{t\to-\infty}\dist(q_{c}(t),\mathcal{A})=\lim_{t\to+\infty}\dist(q_{c}(t),\Omega)=0$.
 \end{itemize}}
\noindent Finally, by {Lemma \ref{lem:sol-W-alphabeta}} we get $\lim_{t\to\pm\infty}\dist(q_{c}(t),\mathcal{V}^{c}_{\pm})=0$, {from which, condition \eqref{eqn:condbrake} is satisfied.} 
\end{itemize}
 
\smallskip

The proof of Theorem~\ref{P:main1} is now completed. 

\begin{Remark}\label{R:min}  Note that, by construction, the solution $q_{c}$ given by Theorem~\ref{P:main1} has a {\sl connecting time interval}
$(\alpha_{q_{c}},\omega_{q_{c}})\subseteq\R$, with $-\infty\leq\alpha_{q_{c}}<\omega_{q_{c}}\leq +\infty$ (coinciding in the statement with  {$(\sigma,\tau)$} in the case~\ref{enum:main:1}, with $(-\infty,\sigma)$ in the case~\ref{enum:main:2} and with $\R$ in the case~\ref{enum:main:3}, respectively), such that
\begin{enumerate}[label={\rm(\arabic*)}, itemsep=3pt]
	\item\label{enum:aposteriori:1} $V(q_{c}(t))>c$ for every $t\in (\alpha_{q_{c}},\omega_{q_{c}})$,
	\item $\lim\limits_{t\to\alpha_{q_{c}}^{+}}\dist(q_{c}(t),\mathcal{V}^{c}_-)=0$, and if $\alpha_{q_{c}}>-\infty$ then $\dot q_{c}(\alpha_{q_{c}})=0$, $V(q_{c}(\alpha_{q_{c}}))=c$ with $q_{c}(\alpha_{q_{c}})\in\mathcal{V}^{c}_-$,
	 \item $\lim\limits_{t\to\omega_{q_{c}}^{-}}\dist(q_{c}(t),\mathcal{V}^{c}_+)=0$, and if $\omega_{q_{c}}<+\infty$ then $\dot q_{c}(\omega_{q_{c}})=0$, $V(q_{c}(\omega_{q_{c}}))= c$ with $q_{c}(\omega_{q_{c}})\in\mathcal{V}^{c}_+$,
	 \item $J_{c,(\alpha_{q_{c}},\omega_{q_{c}})}(q_{c})=m_c$.
\end{enumerate}
Finally, the behavior of $q_{c}$ on $\R$ is obtained by (eventual) reflection and periodic continuation of its restriction to the interval $(\alpha_{q_{c}},\omega_{q_{c}})\subseteq\R$. In particular, we have
\begin{enumerate}[label={\rm(\arabic*)},itemsep=2pt, start=5]
	\item\label{enum:aposteriori:5} $\|q_{c}\|_{L^{\infty}(\R,\R^{N})}=\|q_{c}\|_{L^{\infty}((\alpha_{q_{c}},\omega_{q_{c}}),\R^{N})}$.
\end{enumerate}
\end{Remark}

\section{Some applications}

In this last section we illustrate some applications of Theorem~\ref{P:main1} to certain classes of potentials $V$, extensively studied in the literature. More precisely, we establish existence of connecting orbits to~\eqref{eqn:gradSystem}, in the case of double-well potentials, as well as potentials associated to duffing-like systems and multiple-pendulum-like systems. Additionally, we show in each one of these cases that solutions of heteroclinic type and homoclinic type, at energy level $0$, can be obtained as limits of sequences $(q_c)$ of solutions to~\eqref{eqn:gradSystem}, as $c\to 0^+$. 

\subsection{Double-well potential systems.} As a first example we consider {\sl double-well potential systems} like the ones considered e.g.~in \cite{[BMR]} in the PDE (non-autonomous) setting and in~\cite{alikakos2008connection,antonopoulos2016minimizers,FGN,sternberg2016heteroclinic} in the ODE setting, among others. Precisely, we assume that $V\in C^{1}(\R^{N})$ satisfies
\begin{enumerate}[label={\rm(V\arabic*)},itemsep=2pt]
\item\label{enum:doublewell:V1} There exist $a_{-}\not=a_{+}\in\R^{N}$ such that $V(a_{+})=V(a_{-})=0$, and $V(x)>0$ for $x\in\R^{N}\setminus\{a_{-},a_{+}\}$,
\item\label{enum:doublewell:V2} $\liminf_{|x|\to+\infty}V(x)=:\nu_{0}>0$.\end{enumerate}
\smallskip

As a consequence of Theorem~\ref{P:main1} we have the following result.
\begin{prop}\label{p:twowell} 
Assume that $V\in C^{1}(\R^{N})$ satisfies~\ref{enum:doublewell:V1} and~\ref{enum:doublewell:V2}. If $c\in [0,\nu_{0})$ is such that~\ref{enum:Vc} holds, then the coercivity condition \eqref{eqn:coercivita} of the energy functional $J_c$ over $\Gamma_c$ holds true. In particular, for such value of $c\in [0,\nu_{0})$, Theorem~\ref{P:main1} gives a solution $q_{c}\in C^{2}(\R,\R^N)$ to the problem~\eqref{eqn:gradSystem}-\eqref{eqn:condbrake} that satisfies the pointwise energy constraint {$E_{q_{c}}(t)=-c$ for all $t\in\R$}.
\end{prop}
\begin{proof}[{\bf Proof of Proposition~\ref{p:twowell}}] 
 In order to prove that~\eqref{eqn:coercivita} holds true for $c\in [0,\nu_{0})$ {for which ~\ref{enum:Vc} holds}, we show that there exists $R>0$ such that any minimizing sequence $(q_{n})\subset \Gamma_c$, $J_{c}(q_{n})\to m_{c}= \inf_{\Gamma_{c}}J_{c}(q)$, verifies the uniform bound $\|q_{n}\|_{L^{\infty}(\R,\R^N)}\leq R$.
Arguing by contradiction, let us assume that there is a value $c_*\in [0,\nu_{0})$ for which $({\bf V}^{c_*})$ holds true and that there exists a sequence $(q_{n})\subset \Gamma_{c_*}$ for which $J_{c_*}(q_{n})\to m_{c_*}$ but $\|q_{n}\|_{L^{\infty}(\R,\R^N)}\to +\infty$. \\
Since $c_*<\nu_{0}$ and $\liminf_{|x|\to+\infty}V(x)=\nu_{0}$, we have that, denoting $\mu_{0} :=\frac12(\nu_{0}-c_*)$, there exists $R_{0}>0$ such that
\begin{equation}\label{eq:R0}
V(x)>c_*+\mu_{0}\;\hbox{  if }\;|x|\geq R_{0}.
\end{equation}
Consequently, $\mathcal{V}^{c_*}=\mathcal{V}^{c_*}_{-}\cup \mathcal{V}^{c_*}_{+}\subset  B_{R_{0}}(0)$.  Since $(q_{n})\subset \Gamma_{c_*}$ with $\|q_{n}\|_{L^{\infty}(\R,\R^N)}\to +\infty$, we deduce that for any $R>R_{0}$ there exists $\bar n\in\N$ such that if $n\geq \bar n$ then $q_{n}$ crosses (at least two times) the annulus $B_{R}(0)\setminus B_{R_{0}}(0)$. In particular, by~\eqref{eq:R0}, we obtain that for any $n\geq \bar n$ there is an interval $(s_{n},t_{n})\subset\R$ such that
\[ 
|q_{n}(t_{n})-q_{n}(s_{n})|\geq R-R_{0}\;\,\hbox{ and }\;\, V(q_{n}(t))\geq c_*+\mu_{0}\,\hbox{ for }\,t\in (s_{n},t_{n}).
\]
By Remark \ref{rem:disequazioneF} we then recover that for $n\geq \bar n$
\[
m_{c_*}+o(1)=J_{c_*}(q_{n})\geq \sqrt{2 \mu_{0}}(R-R_{0}).
\]
Since $R$ is arbitrary, the latter contradicts the finiteness of $m_{c_*}$. Indeed, by Lemma \ref{lem:m>0} we know that $m_{c_*}<+\infty$ (in particular, the proof of Lemma \ref{lem:m>0} shows that this does not depend on  \eqref{eqn:coercivita}). 
\end{proof}

\begin{Remark}\label{rem:cbar}
The continuity of $V$ and the assumptions~\ref{enum:doublewell:V1} and~\ref{enum:doublewell:V2} imply that there always exists $\cdwell\in (0,\nu_{0})$ for which the condition~\ref{enum:Vc} is satisfied for every $c$ in the interval $[0,{\cdwell})$. Indeed, as seen in~\eqref{eq:R0}, from~\ref{enum:doublewell:V2} we obtain $\exists R_0>0$ large so that $\mathcal{V}^c\subset B_{R_0}(0)$ for any $c\in(0,\nu_{0})$. Hence, $\mathcal{V}^c$ is compact. Also, in view of~\ref{enum:doublewell:V1} and $V\in C(\R^N)$, {we can choose} $\cdwell$ sufficiently small so that $\mathcal{V}^c$ splits in the disjoint union of two compact sets $\mathcal{V}^{c}_-$ and $\mathcal{V}^{c}_+$, for any $c\in [0,\cdwell)$. In particular,
we can assume that for $\varrho_{0}:=\frac14|a_{-}-a_{+}|$ the condition below holds
\begin{equation}\label{eq:Vpiumeno}
a_{-}\in\mathcal{V}^{c}_-\subset B_{\varrho_{0}}(a_{-})\;\hbox{ and }\;a_{+}\in\mathcal{V}^{c}_+\subset B_{\varrho_{0}}(a_{+}),\;\text{ for all }\,c\in[0,\cdwell),
\end{equation}
by reducing the value of $\cdwell\in (0,\nu_{0})$, if necessary. It follows that $\dist(\mathcal{V}^{c}_-,\mathcal{V}^{c}_+)\geq \frac12|a_{-}-a_{+}|>0$.
\end{Remark}
For any $c\in [0,\nu_{0})$ for which~\ref{enum:Vc} is satisfied, Proposition \ref{p:twowell} provides a solution $q_{c}$ with energy $-c$ which connects $\mathcal{V}^{c}_{-}$ with $\mathcal{V}^{c}_{+}$. 
{As noted in Remark~\ref{rem:regular}}, such a solution is of brake orbit type when $c$ is a regular value for $V$, while it may be of the homoclinic or heteroclinic type if $c$ is a critical value of $V$. Of particular interest is the case when $c=0$, where we see that $\mathcal{V}^{c}_{-}=\{a_{-}\}$ with $\mathcal{V}^{c}_{+}=\{a_{+}\}$ (or viceversa) {and }since $a_{\pm}$ are critical points of $V$, we are in the case (c) of Theorem~\ref{P:main1}. {Therefore} the solution given by Proposition \ref{p:twowell} {for $c=0$} is of heteroclinic type connecting the equilibria $a_{-}$ and $a_{+}$.

\smallskip

We continue our analysis by studying the behavior of the solution $q_c$ given by {Proposition \ref{p:twowell}} as $c\to 0^+$ and we will prove that they converge, in a suitable sense, to a heteroclinic solution connecting the equilibria $a_{\pm}$. Precisely, we have

\begin{prop}\label{p:convergence} Assume that $V\in C^{1}(\R^{N})$ satisfies~\ref{enum:doublewell:V1}-\ref{enum:doublewell:V2} and that $\nabla V$ is locally Lipschitz continuous in $\R^{N}$. Let $c_{n}\to 0^+$ be any sequence and $(q_{c_{n}})$ be the sequence of solutions to the system~\eqref{eqn:gradSystem} given by Proposition \ref{p:twowell}. Then, up to translations and a subsequence, $q_{c_{n}}\to q_{0}$ in $C^{2}_{loc}(\R,\R^{N})$, where $q_{0}$ is a solution to~\eqref{eqn:gradSystem} of heteroclinic type between $a_{-}$ and $a_{+}$, i.e. it satisfies
\[
q_0(t)\to a_\pm\;\text{ and }\;\;\dot q_0(t)\to 0,\;\,\text{ as }\;t\to\pm\infty.
\]
\end{prop}}

\smallskip
{ To prove Proposition \ref{p:convergence}, we begin by establishing in the next lemma a uniform estimate of the $L^{\infty}$-norm of the solutions $(q_{c})$ for $0\leq c<\cdwell$, where $\cdwell$ is given in Remark~\ref{rem:cbar}.}

\begin{Lemma}\label{lem:unif} Assume that $V\in C^{1}(\R^{N})$ satisfies~\ref{enum:doublewell:V1} and~\ref{enum:doublewell:V2} and let  $\cdwell\in(0,\nu_{0})$.
Then there exists $\Rdwell>0$ such that $\|q_{c}\|_{L^{\infty}(\R,\R^{N})}\leq\Rdwell$ for any $c\in[0,\cdwell)$.
\end{Lemma}
\begin{proof} [{\bf Proof of Lemma~\ref{lem:unif}}] First we claim that there exists $\Mdwell>0$ such that
\begin{equation}\label{eq:mc}
m_{c}=\inf_{q\in \Gamma_{c}}J_c(q)\leq \Mdwell\;\hbox{ for any }\;c\in[0,\cdwell).
\end{equation}
Indeed, consider the function $\xi(t) = (1 - t)a_{-} + ta_{+}$ for $t\in [0,1]$. From~\eqref{eq:Vpiumeno} and the compactness of $\mathcal{V}^{c}_{\pm}$ we see that for each $c\in[0,\cdwell)$ there exist $0\leq\sigma_{c}<\tau_{c}\leq 1$  that satisfy 
\[\xi(\sigma_{c})\in \mathcal{V}^{c}_-,\; \xi(\tau_{c})\in\mathcal{V}^{c}_+\;\hbox{ and }\;V(\xi(t))>c\;\hbox{ for any }\,t\in
(\sigma_{c},\tau_{c}).\]
Then, the function 
\[q_{c,\xi}(t)=\begin{cases}\xi(\sigma_{c})&\text{if }\, t\leq \sigma_{c},\\
\xi(t)&\text{if }\,\sigma_{c}<t<\tau_{c},\\
\xi(\tau_{c})&\text{if }\, \tau_{c}\leq t.\end{cases}
\]
belongs to $\Gamma_{c}$, and so  our claim~\eqref{eq:mc} follows by a plain estimate:
\[m_{c}\leq J_{c}(q_{c,\xi})\leq \tfrac 12|a_{+}-a_{-}|^{2}+\max_{s\in [0,1]}V(\xi(s))=: \Mdwell.\]
If $q_{c}$ denotes the solution in~Proposition~\ref{p:twowell}, corresponding to a $c\in [0,\cdwell)$, by Remark \ref{R:min} there is a {\sl connecting time interval}
$(\alpha_{q_{c}},\omega_{q_{c}})\subset\R$
for which properties~\ref{enum:aposteriori:1}-\ref{enum:aposteriori:5} hold true. In particular, by~\eqref{eq:mc}
\begin{equation}\label{eq:funz}
J_{c,(\alpha_{q_{c}},\omega_{q_{c}})}(q_{c})=m_{c}\leq \Mdwell\;\,\hbox{ for any }\,c\in [0,\cdwell),
\end{equation}
holds true. We now claim that there exists $\Rdwell>0$ in such a way that
\begin{equation}\label{eq:norm}
\|q_{c}\|_{L^{\infty}((\alpha_{q_{c}},\omega_{q_{c}}),\R^{N})}\leq\Rdwell\;\;\hbox{ for any }\,c\in [0,\cdwell).
\end{equation}
The proof of Lemma~\ref{lem:unif} will be concluded upon establishing~\eqref{eq:norm}, since condition~\ref{enum:aposteriori:5} in Remark~\ref{R:min} gives $\|q_{c}\|_{L^{\infty}((\alpha_{q_{c}},\omega_{q_{c}}),\R^{N})}=\|q_{c}\|_{L^{\infty}(\R,\R^{N})}$. To prove \eqref{eq:norm}, arguing by contradiction, we assume that for every $R>0$ there exists $c_R\in [0,\cdwell)$ such that 
\begin{equation}\label{eq:contradassumption}
\|q_{c_R}\|_{L^{\infty}((\alpha_{q_{c_R}},\omega_{q_{c_R}}),\R^{N})}>R. 
\end{equation}
Let us observe that assumption~\ref{enum:doublewell:V2} ensures that for $h_0:=\frac12(\nu_{0}-\cdwell)$ there exists $R_{0}>0$ for which 
\begin{equation*}\label{eq:newR0}
V(x)>\cdwell+h_{0}\;\hbox{  if }\;|x|\geq R_{0}.
\end{equation*}
In particular, {this} shows that $\mathcal{V}^{c}=\mathcal{V}^{c}_{-}\cup \mathcal{V}^{c}_{+}\subset B_{R_{0}}(0)$ for any $c\in [0,\cdwell)$. In the remaining of the proof we choose $R\in(R_{0},+\infty)$ to satisfy $\sqrt{2 h_{0}}(R-R_{0})>\Mdwell$. For this choice of $R$, the contradiction assumption~\eqref{eq:contradassumption} implies that the trajectory of $q_{c_R}\in\mathcal{M}_{c_R}$ crosses the annulus $B_{R}(0)\setminus B_{R_{0}}(0)$; thus, there is an interval $[\sigma,\tau]\subset (\alpha_{q_{c_R}},\omega_{q_{c_R}})$ in such a way that
\[ 
|q_{c_R}(\tau)-q_{c_R}(\sigma)|\geq R-R_{0}\;\;\hbox{ and }\;\;V(q_{c_R}(t))\geq \cdwell+h_{0}\,\hbox{ for any }\,t\in (\sigma,\tau).
\]
Hence, Remark~\ref{rem:disequazioneF} along with the properties above, show the strict lower bound on the energy of $q_{c_R}$
\[
m_{c_R}=J_{c_R,(\alpha_{q_{c_R}},\omega_{q_{c_R}})}(q_{c_R})\geq \sqrt{2 h_{0}}(R-R_{0})>\Mdwell.
\]
However, this contradicts the upper bound~\eqref{eq:funz}. In this way, we have argued that~\eqref{eq:norm} follows, which in turn completes the proof of this lemma.
\end{proof}

{We can now prove Proposition~\ref{p:convergence}}. Without loss of generality, let the sequence $c_{n}\to 0^+$ be so that $(c_{n})\subset (0,\cdwell)$. Since $c_{n}<\cdwell$, we know from Remark~\ref{rem:cbar} that $(\mathbf{V}^{c_n})$ holds true for all $n\in\N$. By Remark~\ref{R:min}, for each $n\in\N$, the solution $q_{c_{n}}$ given by  Proposition \ref{p:twowell} has a {\sl connecting time interval}
$(\alpha_{n},\omega_{n})\subset\R$, with $-\infty\leq\alpha_{n}<\omega_{n}\leq +\infty$, in such a way that
\begin{enumerate}[label={\rm$(\arabic*_n)$},itemsep=3pt, leftmargin=2.5em]
	\item\label{enum:1n} $V(q_{c_n}(t))>c_{n}$ for every $t\in (\alpha_{n},\omega_{n})$,
	\item\label{enum:2n} $\lim\limits_{t\to\alpha_{n}^{+}}\dist(q_{c_n}(t),\mathcal{V}^{c_n}_-)=0$, and if $\alpha_{n}>-\infty$ then $\dot q_{c_n}(\alpha_{n})=0$, $V(q_{c_n}(\alpha_{n}))=c_{n}$ with $q_{c_n}(\alpha_{n})\in\mathcal{V}^{c_{n}}_-$,
	 \item\label{enum:3n} $\lim\limits_{t\to\omega_{n}^{-}}\dist(q_{c_n}(t),\mathcal{V}^{c_{n}}_+)=0$, and if $\omega_{n}<+\infty$ then $\dot q_{c_n}(\omega_{n})=0$, $V(q_{c_n}(\omega_{n}))= c_{n}$ with $q_{c_n}(\omega_{n})\in\mathcal{V}^{c_{n}}_+$.
	 \item\label{enum:4n}  $J_{c_{n},(\alpha_{n},\omega_{n})}(q_{c_n})=m_{c_n}=\inf\limits_{q\in\Gamma_{c_{n}}}J_{c_{n}}(q)$. 
\end{enumerate}

We first start by renormalizing the sequence $(q_{c_n})$ using the following phase shift procedure. In light of the properties~\ref{enum:2n} and~\ref{enum:3n}, for any $n\in\N$ there exists $\zeta_{n}\in (\alpha_{n},\omega_{n})$ such that
\[
\dist(q_{c_n}(\zeta_{n}),\{a_{-},a_{+}\})=\varrho_{0},
\]
for $\varrho_{0}:=\frac14|a_--a_+|$ as in~\eqref{eq:Vpiumeno}. {Hence,} up to translations, eventually renaming $q_{c_n}$ to be $q_{c_n}(\cdot-\zeta_{n})$,  we can assume
\begin{equation}\label{eq:trasl}
	\alpha_{n}<0<\omega_{n}\;\hbox{ and }\;\dist(q_{c_n}(0),\{a_{-},a_{+}\})=\varrho_{0},\;\hbox{ for any }n\in\N.
\end{equation}
We now argue that $(q_{c_n})$ converges, in the $C^2$-topology on compact sets, to an entire solution $q_0$ of the system $\ddot q=\nabla V(q)$ over $\R$. To see this, let us observe that $(q_{c_n})$, $(\dot q_{c_n})$ and $(\ddot q_{c_n})$ are uniformly bounded in $L^{\infty}(\R,\R^{N})$, given the fact that $\|q_{c_n}\|_{L^{\infty}(\R,\R^{N})}\leq \Rdwell$ for any $n\in\N$ (by Lemma~\ref{lem:unif}). More precisely, for every $n\in\N$ we have the bounds
\begin{equation*}\label{eq:normqn}
\begin{array}{l}
\|\dot q_{c_n}\|_{L^{\infty}(\R,\R^{N})}\leq C_{dw}:=(2\max\{V(x):|x|\leq\Rdwell\})^{1/2}<+\infty,\;\text{ and}\\[.5em]
\|\ddot q_{c_n}\|_{L^{\infty}(\R,\R^{N})}\leq C'_{dw}:=\max\{|\nabla V(x)|:|x|\leq\Rdwell\}<+\infty,
\end{array}
\end{equation*}
where the former is a consequence of the pointwise energy constraint $E_{q_{c_n}}(t)=-c_{n}$ for all $t\in\R$, and the latter follows since $\ddot q_{c_n}=\nabla V(q_{c_n})$ on $\R$. 

An application of the Ascoli-Arzel\`a Theorem shows that there exists $q_{0}\in C^{1}(\R,\R^{N})$ and a subsequence of $(q_{c_n})$, still denoted $(q_{c_n})$, such that 
\[
q_{c_n}\to q_{0}\;\hbox{ in }\;C_{loc}^{1}(\R,\R^{N}),\;\hbox{ as }\;n\to +\infty.
\]
In particular,~\eqref{eq:trasl} implies that
\begin{equation}\label{eq:distq0}
\dist(q_0(0),\{a_{-},a_{+}\})=\varrho_{0}.
\end{equation}
Moreover, the above convergence can be improved to $q_{c_n}\to q_{0}$ in $C_{loc}^{2}(\R,\R^{N})$ as $n\to +\infty$, by using once more that $\ddot q_{c_n}=\nabla V(q_{c_n})$ on $\R$, for any $n\in\N$. In fact, the latter convergence shows, in turn, that 
\begin{equation}\label{eq:eqq0} 
\ddot q_{0}=\nabla V(q_{0})\;\hbox{ on }\;\R.
\end{equation}
To conclude the proof of Proposition~\ref{p:convergence}, it will be enough to establish
\begin{equation}\label{eq:q0hetero}
q_{0}(t)\to a_{\pm}\;\hbox{ as }\;t\to \pm\infty.
\end{equation}  
Indeed, once the validity of~\eqref{eq:q0hetero} has been proved, we then get $\ddot q(t)\to 0$ as $t\to \pm\infty$ from~\eqref{eq:eqq0}, which in turn shows that $\dot q(t)\to 0$ as $t\to \pm\infty$ by interpolation inequalities. \smallskip

To prove \eqref{eq:q0hetero}, let us first note that conditions \eqref{eq:trasl},~\eqref{eq:distq0} and~\eqref{eq:eqq0} imply altogether
\begin{equation}\label{eq:infty}
\alpha_{n}\to -\infty\;\hbox{ and }\;\omega_{n}\to+\infty,\;\hbox{ as }n\to+\infty.
\end{equation}
Indeed, arguing by contradiction, assume that along a subsequence $\alpha_{n}$ is bounded. As $\alpha_n<0$ then, up to a subsequence, we deduce that $\alpha_{n}\to \alpha_{0}$ for some $\alpha_0\leq 0$. 
By ($2_{n}$) and \eqref{eq:Vpiumeno} we have $q_{c_n}(\alpha_{n})\in \mathcal{V}^{c_{n}}_{-}\subset B_{\varrho_{0}}(a_{-})$, which combined with $c_{n}\to 0$,~\ref{enum:doublewell:V1} and~\ref{enum:doublewell:V2} then yields $q_{c_n}(\alpha_{n})\to a_{-}$. This, $\dot q_{c_n}(\alpha_{n})=0$ for any $n\in\N$, and the fact that $q_{c_n}\to q_{0}$ in $C^{1}_{loc}(\R,\R^{N})$ allow us to conclude  $q_{0}(\alpha_{0})=a_{-}$ and $\dot q_{0}(\alpha_{0})=0$.
Nonetheless, the uniqueness of solutions to the Cauchy problem would then imply that $q_{0}(t)=a_{-}$ for any $t\in\R$. This is contrary to~\eqref{eq:distq0}, thus showing $\alpha_{n}\to-\infty$, as claimed. Analogously, one can prove that $\omega_{n}\to+\infty$. Therefore,~\eqref{eq:infty} follows.\smallskip

In order to establish \eqref{eq:q0hetero}, let us observe that it is sufficient to prove:
for any $r\in (0,\rho_{0})$ there exists $L_{r}^\pm>0$ and $n_{r}\geq\bar n$ such that for any $n\geq n_{r}$, there hold
\begin{equation}\label{eq:trappingqn} 
|q_{c_n}(t)-a_{-}|<r\;\hbox{ for }\;t\in (\alpha_{n},-L_{r}^-),\;\hbox{ and }\;\; |q_{c_n}(t)-a_{+}|<r\;\hbox{ for }\;t\in (L_{r}^+,\omega_{n}).\end{equation}

Indeed, by taking the limit as $n\to+\infty$ in \eqref{eq:trappingqn}, we then conclude in view of the pointwise convergence $q_{c_n}\to q_{0}$ and~\eqref{eq:infty},  that for any $r\in (0,\varrho_{0})$ there exists $L_{r}:=\max\{L^{-}_{r},L^{+}_{r}\}$ such that
\[ 
|q_{0}(t)-a_{-}|<r\,\hbox{ for }\,t\in (-\infty,-L_{r}),\;\hbox{ and }\;|q_{0}(t)-a_{+}|<r\,\hbox{ for }\,t\in (L_{r},+\infty),
\]
therefore,~\eqref{eq:q0hetero} follows.\smallskip

{To obtain the first estimate in \eqref{eq:trappingqn} we argue by contradiction assuming that there exists $\bar r\in (0,\rho_{0})$, a subsequence of $(q_{c_n})$, still denoted $(q_{c_n})$, and a sequence $s_{n}\to-\infty$ such that for any $n\in\N$
\begin{equation}\label{eq:sn} 
|q_{c_n}(s_{n})-a_{-}|>\bar r\;\text{ with }\;
s_{n}\in (\alpha_{n},0).
\end{equation}
Then, by observing that~\ref{enum:4n} and~\eqref{eq:funz} imply the inequality 
\[
\Mdwell\geq m_{c_{n}}=J_{c_{n},(\alpha_{n},\omega_{n})}(q_{c_n})\geq \int_{s_{n}}^{0}(V(q_{c_n}(t))-c_{n})\,dt\geq
|s_{n}|\inf_{t\in(s_{n},0)}(V(q_{c_n}(t))-c_{n}),
\]
we obtain that the contradiction hypothesis $s_{n}\to-\infty$ yields
\[
\inf_{t\in(s_{n},0)}V(q_{c_n}(t))\leq c_{n}+\tfrac{\Mdwell}{|s_{n}|}\to 0\;\;\hbox{ as }\;n\to+\infty.
\]
In particular, we deduce that for every $n\in\N$ there exists $t_{n}\in (s_{n},0)$ so that $V(q_{c_n}(t_{n}))\to 0\hbox{ as }n\to+\infty$. This, in light of~\ref{enum:doublewell:V1} and~\ref{enum:doublewell:V2}, shows 
\begin{equation}\label{eq:distamenoapiu*}
\dist(q_{c_n}(t_{n}),\{a_{-},a_{+}\})\to 0\;\hbox{ as }\;n\to +\infty.
\end{equation}
Let us show, if fact, that 
\begin{equation}\label{eq:toameno*}\liminf\limits_{n\to+\infty}\dist(q_{c_n}(t_{n}),a_{+})>0.
\end{equation}
Indeed, if \eqref{eq:toameno*} fails, along a subsequence (still denoted $q_{c_n}$) we have $q_{c_n}(t_{n})\to a_{+}$.
We connect the point $q_{c_n}(t_{n})$ with $a_{+}$ with the segment $\{(1-\sigma)q_{c_n}(t_{n})+\sigma a_{+}:\,
\sigma\in [0,1]\}$. By continuity, since $V(q_{c_n}(t_{n}))>c_{n}>0=V(a_{+})$, there is $\sigma_{n}\in (0,1)$ such that
\begin{equation}\label{eq:segment}V((1-\sigma_{n})q_{c_n}(t_{n})+\sigma_{n} a_{+})=c_{n}\;\,\hbox{ and }\;\,V((1-\sigma)q_{c_n}(t_{n})+\sigma a_{+})>c_{n}\,\hbox{ for any }\,\sigma\in [0,\sigma_{n}).\end{equation}
In particular $(1-\sigma_{n})q_{c_n}(t_{n})+\sigma_{n} a_{+}\in\mathcal{V}^{c_{n}}_{+}$ and the function
\[q_{-,n}(t)=\begin{cases} q_{c_n}(\alpha_{n})&\text{if }\,t<\alpha_{n},\\
q_{c_n}(t)&\text{if }\,\alpha_{n}\leq t< t_{n},\\
(1-(t-t_{n}))q_{c_n}(t_{n})+(t-t_{n})a_{+}&\text{if }\,t_{n}\leq t<t_{n}+\sigma_{n},\\
(1-\sigma_{n})q_{c_n}(t_{n})+\sigma_{n}a_{+}&\text{if }\,t_{n}+\sigma_{n}\leq t.\end{cases}\]
(where we agree to omit the first item in the definition  if $\alpha_{n}=-\infty$) is by construction an element of $\Gamma_{c_{n}}$  for any $n\in\N$, whence $J_{c_{n}}(q_{-,n})\geq m_{c_{n}}$. Since $q_{c_n}(t_{n})\to a_{+}$ we have $V((1-\sigma)q_{c_n}(t_{n})+\sigma a_{+})\to 0$ as $n\to +\infty$ uniformly for $\sigma\in [0,1]$. Thus, we derive that,
\[J_{c_{n},(t_{n},+\infty)}(q_{-,n})=J_{c_{n},(t_{n},t_{n}+\sigma_{n})}(q_{-,n})=\tfrac{\sigma_{n}}2|q_{c_n}(t_{n})-a_{+}|^{2}+\int_{t_{n}}^{t_{n}+\sigma_{n}}\bigl(V((1-\sigma)q_{c_n}(t_{n})+\sigma a_{+})-c_{n}\bigr)d\sigma\to 0,
\]
as $n\to+\infty$.
Since $q_{-,n}=q_{c_n}$ on $(\alpha_{n},t_{n})$ and since $q_{-,n}$ is constant outside $(\alpha_{n},t_{n}+\sigma_{n})$, the latter shows
\[J_{c_{n},(\alpha_{n},t_{n})}(q_{c_n})=J_{c_{n},(\alpha_{n},t_{n})}(q_{-,n})=
J_{c_{n}}(q_{-,n})-J_{c_{n},(t_{n},+\infty)}(q_{-,n})\geq m_{c_{n}}-o(1).\]
This bound with $(4_n)$ implies that
\[
m_{c_{n}}=J_{c_{n},(\alpha_{n},\omega_{n})}(q_{c_n})=J_{c_{n},(\alpha_{n},t_{n})}(q_{c_n})+J_{c_{n},(t_{n},\omega_{n})}(q_{c_n})\geq m_{c_{n}}-o(1)+J_{c_{n},(t_{n},\omega_{n})}(q_{c_n}),
\]
and so
\begin{equation}\label{eq:coda*}
J_{c_{n},(t_{n},\omega_{n})}(q_{c_n})\to 0\;\hbox{ as }\;n\to+\infty.
\end{equation}
On the other hand, since $q_{c_n}(t_{n})\to a_{+}$ and, by \eqref{eq:trasl}, $\dist(q_{c_n}(0),\{a_{-},a_{+}\})=\varrho_{0}$, when $n$ is large the trajectory of $q_{c_n}$ crosses the annulus $B_{\varrho_{0}/2}(a_{+})\setminus B_{\varrho_{0}/4}(a_{+})$ at least one in the interval $(t_{n},0)$, i.e., there exists $(t_{1,n},t_{2,n})\subset (t_{n},0)$ such that $|q_{c_n}(t_{2,n})-q_{c_n}(t_{1,n})|=\tfrac{\varrho_{0}}4$ and $\dist(q_{c_n}(t),\{a_{-},a_{+}\})\geq \tfrac{\varrho_{0}}4$ for any $t\in (t_{1,n},t_{2,n})$. If we set $\mu(\tfrac{\varrho_{0}}4):=\inf_{\R^{N}\setminus (B_{\varrho_{0}/4}(a_{-})\cup B_{\varrho_0/4}(a_{+}))}V$ it follows that for every $t\in (t_{1,n},t_{2,n})$, $V(q_{c_n}(t))\geq \mu(\tfrac{\varrho_{0}}4)$, and hence Remark \ref{rem:disequazioneF} yields the bound
\[J_{c_{n},(t_{n},\omega_{n})}(q_{c_n})\geq J_{c_{n},(t_{1,n},t_{2,n})}(q_{c_n})\geq\sqrt{2(\mu(\tfrac{\varrho_{0}}4)-c_{n})}|q_{c_n}(t_{2,n})-q_{c_n}(t_{1,n})|>\sqrt{\mu(\tfrac{\varrho_{0}}4)}\tfrac{\varrho_{0}}4\]
 for $n$ large enough. This last inequality contradicts \eqref{eq:coda*} proving \eqref{eq:toameno*}.\smallskip

By \eqref{eq:distamenoapiu*} and \eqref{eq:toameno*} we obtain $q_{c_n}(t_{n})\to a_{-}$. We now show that this case is not possible either, thus obtaining a contradiction with \eqref{eq:distamenoapiu*}. This will establish the first estimate of \eqref{eq:trappingqn}.\smallskip 

To prove that $q_{c_n}(t_{n})\to a_{-}$ cannot occur, we use an argument similar to the one used above. If $q_{c_n}(t_{n})\to a_{-}$, we fix $\sigma_{n}\in (0,1)$ such that
\begin{equation*}\label{eq:segmentdue}
V((1-\sigma_{n}) a_{-}+\sigma_{n} q_{c_n}(t_{n}))=c_{n}\;\;\hbox{ and }\;\;V((1-\sigma) a_{-}+\sigma q_{c_n}(t_{n}))>c_{n}\hbox{
for any }\sigma\in (\sigma_{n},1].
\end{equation*}
Then the function 
\[
q_{+,n}(t):=\begin{cases} (1-\sigma_{n}) a_{-}+\sigma_{n} q_{c_n}(t_{n})&\text{if }\,t\leq t_{n}-(1-\sigma_{n}),\\
(t_{n}-t) a_{-}+(t-t_{n}+1) q_{c_n}(t_{n})&\text{if }\,t_{n}-(1-\sigma_{n})<t<t_{n},\\
q_{c_n}(t)&\text{if }\,t_{n}<t\leq \omega_{n},\\
q_{c_n}(\omega_{n})&\text{if }\,\omega_{n}<t.
\end{cases}
\]
(where we omit the last item if $\omega_{n}=+\infty$) belongs to $\Gamma_{c_{n}}$ and $J_{c_{n}}(q_{+,n})\geq m_{c_{n}}$. Since $q_{c_n}(t_{n})\to a_{-}$, with an argument analogous to the one used for $q_{-,n}$, we obtain
\[
J_{c_{n},(-\infty,t_{n})}(q_{+,n})=J_{c_{n},(t_{n}-(1-\sigma_{n}),t_{n})}(q_{+,n})\to 0\;\hbox{ as }\;n\to+\infty.
\]
A reasoning similar to the one that lead to \eqref{eq:coda*}, then shows
\begin{equation}\label{eq:coda2}
J_{c_{n},(\alpha_{n},t_{n})}(q_{c_n})\to 0\;\hbox{ as }\;n\to+\infty.
\end{equation}
By \eqref{eq:sn} and \eqref{eq:toameno*}, we are now in the situation where
\[\alpha_{n}<s_{n}<t_{n}<0,\; |q_{c_n}(s_{n})-a_{-}|>\bar r,\,\hbox{ and }\;q_{c_n}(t_{n})\to  a_{-}\text{ as }\, n\to+\infty.\]
Then, when $n$ is large, the trajectory of $q_{c_n}$ crosses the annulus $B_{\bar r/2}(a_{-})\setminus B_{\bar r/4}(a_{-})$  in the interval $(s_{n},t_{n})$ and so there exists $(t_{3,n},t_{4,n})\subset (s_{n},t_{n})$ such that $|q_{c_n}(t_{3,n})-q_{c_n}(t_{4,n})|=\tfrac{\bar r}4$ and $\dist(q_{c_n}(t),\{a_{-},a_{+}\})\geq \tfrac{\bar r}4$ for $t\in (t_{3,n},t_{4,n})$. As above, since $V(q(t))\geq \mu(\tfrac{\bar r}4):=\inf_{\R^{N}\setminus (B_{\bar r/4}(a_{-})\cup B_{\bar r/4}(a_{+}))}V$ for any $t\in (t_{3,n},t_{4,n})$, we can use Remark \ref{rem:disequazioneF} to obtain
\[J_{c_{n},(\alpha_{n},t_{n})}(q_{c_n})\geq J_{c_{n},(t_{3,n},t_{4,n})}(q_{c_n})>\sqrt{\mu(\tfrac{\bar r}4)}\tfrac{\bar r}4\;\;\hbox{ for }n\hbox{ large}.\]
This contradicts \eqref{eq:coda2} and so the case $q(t_{n})\to a_{-}$ cannot occur either. Then the first estimate of \eqref{eq:trappingqn} follows.\\
One can readily verify that a strategy similar as the one above can be used to establish the second estimate of \eqref{eq:trappingqn}. In conclusion, \eqref{eq:trappingqn} follows, and as a consequence~\eqref{eq:q0hetero} has been established. The proof of Proposition~\ref{p:convergence} is now complete.}

\medskip

\subsection{Duffing like systems.} A second application of Theorem~\ref{P:main1} include {\sl Duffing like} systems. More precisely, we follow the assumptions made in~\cite{ambrosetti1992homoclinics} and~\cite{rabinowitz1991results}: let $V$ be a $C^{1}(\R^{N})$ potential satisfying
\begin{enumerate}[label={\rm(V\arabic*)},start=3,itemsep=2pt]
\item\label{enum:duffing:V3} $V$ has a strict local minimum at $x_{0}:=0$, with value $V(0)=0$:
\[
\exists r_0>0\hbox{ such that }V(x)>0\hbox{ for any }x\in \overline{B_{4r_0}(0)}\setminus\{0\};
\]
\item\label{enum:duffing:V4} The set $\mathcal{C}_{0}:=\{ x\in\R^{N}: V(x)> 0\}\cup \{ 0\}$ is bounded, and such that $\nabla V(x)\neq 0$ for any $x\in\partial\mathcal{C}_{0}$.
\end{enumerate}
\smallskip

We observe that {for any $c\ge 0$}, $J_{c}$ satisfies the coercivity property~\eqref{eqn:coercivita} on $\Gamma_{c}$. To see this, in view of~\ref{enum:duffing:V3} and~\ref{enum:duffing:V4}, note that {$\{ x\in\R^{N}:V(x)\ge c\}\subset\overline{\mathcal{C}_{0}}$} thus, by {definition of $\Gamma_c$}, $q(\R)\subset\overline{\mathcal{C}_{0}}$ for any $q\in\Gamma_c$. Hence, since $\mathcal{C}_{0}$ is a bounded set, we conclude that \eqref{eqn:coercivita} holds.
\medskip

This discussion shows that Theorem~\ref{P:main1} applies, and so we have
\begin{prop}\label{p:duffing} 
Assume that $V\in C^{1}(\R^{N})$ satisfies~\ref{enum:duffing:V3} and~\ref{enum:duffing:V4}. If $c\geq 0$ is such that~\ref{enum:Vc} holds true then Theorem~\ref{P:main1} gives a solution $q_{c}\in C^{2}(\R,\R^N)$ to the problem~\eqref{eqn:gradSystem}-\eqref{eqn:condbrake} with energy $E_{q_{c}}(t)=-c$ for all $t\in\R$.
\end{prop}

We now remark that condition~\ref{enum:Vc} is verified if $c>0$ is chosen sufficiently small.
Indeed, by~\ref{enum:duffing:V3} there exists $\nu_{1}>0$ such that $V(x)\geq \nu_{1}$ for $|x|=r_0$. Then $\mathcal{V}^{c}\cap\partial B_{r_0}(0)=\emptyset$ for any $c\in [0,\nu_{1})$ and~\ref{enum:Vc} is satisfied by the sets 
\begin{equation}\label{eq:separationD}
\mathcal{V}^{c}_-:= \mathcal{V}^{c}\cap B_{r_0}(0)\;\;\hbox{ and }\;\;\mathcal{V}^{c}_+:= \mathcal{V}^{c}\setminus B_{r_0}(0).
\end{equation} 
In particular, the fact that $\dist(\mathcal{V}^{c}_{-},\mathcal{V}^{c}_{+})>0$ is guaranteed by the continuity of $V$.

It is worth distinguishing among the different type of solutions $q_{c}$ that we can obtain from Proposition \ref{p:duffing}, for suitable choices of $c\in[0,+\infty)$. In the case $c=0$, we have from~\eqref{eq:separationD} that {$\mathcal{V}^{c}_-=\{0\}$} and $\mathcal{V}^{c}_+=\R^N\setminus\mathcal{C}_{0}$.  Then Proposition~\ref{p:duffing} above states the existence of a solution $q_{0}$ to~\eqref{eqn:gradSystem} connecting {${0}$} with
$\partial \mathcal{C}_{0}$. This solution {\em cannot} be of {\sl brake orbit} type since $\nabla V({0})=0$, and so by Theorem~\ref{P:main1}-\ref{enum:main:1} it cannot have a {\sl contact} point with $\mathcal{V}^{c}_-$.
Analogously, $q_{0}$ {\em cannot} be of {\sl heteroclinic} type, since in this case, by Theorem~\ref{P:main1}-\ref{enum:main:3}, there must exist a set $\Omega\subset \partial\mathcal{C}_{0}$ consisting of critical points of $V$; thus contradicting the hypothesis made on $\partial\mathcal{C}_{0}$ in~\ref{enum:duffing:V4}.
We conclude that $q_{0}$ must be a {\sl homoclinic type} solution with energy $E_{q_{0}}(t)=0$ for all $t\in\R$ satisfying $q_0(t)\to 0$, $\dot q_0(t)\to 0$ as $t\to\pm\infty$, and that reaches $\partial\mathcal{C}_{0}$ at a {\sl contact} time $\sigma\in\R$, with respect to which it is symmetric (see Figure~\ref{fig:duffing}(a)). This gives back the result already proved in~\cite{ambrosetti1992homoclinics},~\cite{antonopoulos2016minimizers},\cite{fusco2017existence} and~\cite{rabinowitz1991results}. In contrast, let us now consider the cases $0<c<\nu_{1}$, where $\nu_{1}:=\min\{V(x):|x|=r_0\}$ as above. Relative to the separation property \eqref{eq:separationD}, Proposition \ref{p:duffing} provides for any such value $c$ the existence of a connecting orbit $q_{c}$ between $\mathcal{V}^{c}_{-}$ and $\mathcal{V}^{c}_{+}$. Arguing as above, we recognize that 
{\em if $c$ is a regular value of} $V$ then $q_{c}$ is a {\sl brake orbit type} solution  of~\eqref{eqn:gradSystem} with energy $E_{q_{c}}(t)=-c$ for all $t\in\R$, connecting $\mathcal{V}^c_-$ and $\mathcal{V}^{c}_+$  (see Figure~\ref{fig:duffing}(b)). 

\begin{figure}[!h]
\def\svgwidth{0.9\columnwidth}
\begin{center} 
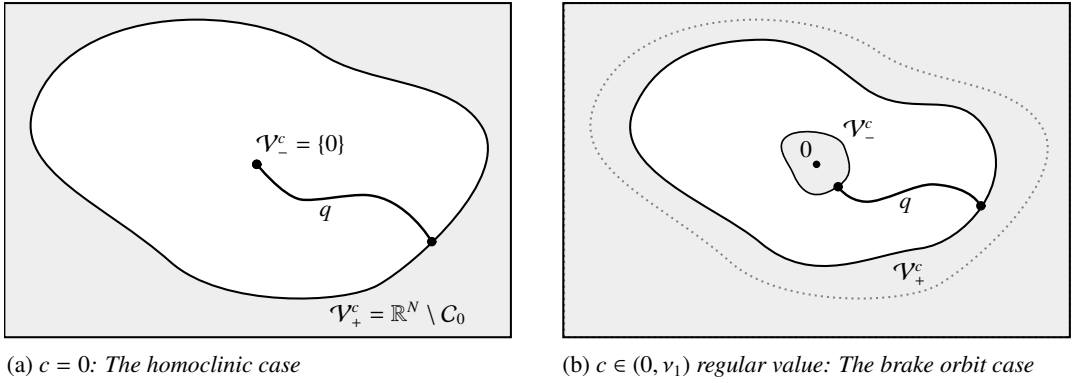
\end{center}
 \caption{Possible configurations in Duffing like systems.}\label{fig:duffing}
\end{figure} 

\smallskip

To analyze the case where $c$ is not a regular value of $V$, we use an approximation argument. Let us first note that $\dist(\mathcal{V}^{c}_{+},\partial\mathcal{C}_{0})\to 0$ as $c\to 0^{+}$ in view of compactness of these sets, and the continuity of $V$. Since $\nabla V\not=0$ on $\partial\mathcal{C}_{0}$,  the  continuity of $\nabla V$ and compactness of $\mathcal{V}^{c}_{+}$ show  $\nabla V\not=0$ on $\partial \mathcal{V}^{c}_+$ if $c$ is small,  say $c<\cduff$ for some number $\cduff\in (0,\nu_{1})$. So by arguing as above, the solution $q_{c}$ can be either of brake orbit type or of homoclinic type (depending on whether $\nabla V\not=0$ on $\partial\mathcal{V}^{c}_{-}$), reaching $\partial \mathcal{V}^{c}_+$ at a {\sl contact} time with respect to which it is symmetric. Using Remark~\ref{R:min} we conclude that  for any $c\in (0,\cduff)$ the solution $q_{c}$ given by Proposition \ref{p:duffing} relative to the decomposition \eqref{eq:separationD} has a {\sl connecting time interval}
$(\alpha_{q_{c}},\omega_{q_{c}})\subset\R$, with $-\infty\leq\alpha_{q_{c}}<\omega_{q_{c}}< +\infty$, 
 such that
\begin{enumerate}[label={\rm$(\arabic*_{\it D})$}, itemsep=3pt, leftmargin=3em]
	\item\label{enum:1cduff} $V(q_{c}(t))>c$ for every $t\in (\alpha_{q_{c}},\omega_{q_{c}})$,
	\item\label{enum:2cduff} $\lim\limits_{t\to\alpha_{q_{c}}^{+}}\dist(q_{c}(t),\mathcal{V}^{c}_-)=0$, and if $\alpha_{q_{c}}>-\infty$ then $\dot q_{c}(\alpha_{q_{c}})=0$,
	$V(q_{c}(\alpha_{q_{c}}))=c$ with $q_{c}(\alpha_{q_{c}})\in\mathcal{V}^{c}_-$,
	 \item\label{enum:3cduff}  $\dot q_{c}(\omega_{q_{c}})=0$, $V(q_{c}(\omega_{q_{c}}))= c$ with $q_{c}(\omega_{q_{c}})\in\mathcal{V}^{c}_+$, 
	 \item\label{enum:4cduff} $J_{c,(\alpha_{q_{c}},\omega_{q_{c}})}(q_{c})=m_{c}$.
\end{enumerate}

We continue by observing that
\begin{Remark}\label{R:boundmc}  There exists $\Mduff>0$ such that
\begin{equation*}\label{eq:mcduffing}
m_{c}=\inf_{\Gamma_{c}}J_{c}\leq \Mduff\,\;\hbox{ for any }\,c\in[0,\cduff).
\end{equation*}

This is obtained along the lines of the proof of \eqref{eq:mc}. Let us fix $\zeta\in\R^{N}$ with $|\zeta|=1$, then we may consider the ray $\{t\zeta:t\geq 0\}\subset\R^N$. From~\ref{enum:duffing:V3}, there must exist $\bar t>0$ such that $\{
t\zeta:0\leq t<\bar t\}\subset\mathcal{C}_{0}$, while $\bar t\zeta\in\partial\mathcal{C}_{0}$. Moreover
for any $c\in[0,\cduff)$ there exist
$0\leq\sigma_{c}<\tau_{c}\leq \bar t$ such that 
\[
\sigma_{c}\zeta\in \mathcal{V}^{c}_-,\; \tau_{c}\zeta\in \mathcal{V}^{c}_+\;\hbox{ and }\;V(t\zeta)>c\;\hbox{ for any }t\in
(\sigma_{c},\tau_{c}).
\]
As in the proof of~\eqref{eq:mc}, we readily see the function 
\[
q_{c,\zeta}(t):=\begin{cases}\sigma_{c}\zeta&\text{if }\, t < \sigma_{c},\\
t\zeta& \text{if }\, \sigma_{c}\leq t<\tau_{c},\\
\tau_{c}\zeta& \text{if }\; \tau_{c}\leq t.\end{cases}
\]
belongs to $\Gamma_{c}$. Hence, setting $\diam(\mathcal{C}_{0}):=\sup\{|x-y|\,: x,y \in \mathcal{C}_{0}\}$, we obtain that
\[
\sup_{c\in[0,\cduff)}m_{c}\leq \sup_{c\in[0,\cduff)}J_{c}(q_{c,\zeta})
\leq {\diam(\mathcal{C}_{0})\bigl(\tfrac 12+\max_{x\in\mathcal{C}_{0}}V(x)\bigr)=: \Mduff.}
\]
\end{Remark}

Arguments similar to the ones used in the case of double well potentials show that the solutions $(q_{c})$ accumulate a near solution of homoclinic type, as $c$ approaches zero.
\begin{prop}\label{p:convergenceduffing} 
Assume that $V\in C^{1}(\R^{N})$ satisfies~\ref{enum:duffing:V3}-\ref{enum:duffing:V4}, and that $\nabla V$ is locally Lipschitz continuous in $\R^{N}$. For any sequence $c_{n}\to 0^+$, consider the sequence of solutions $(q_{c_n})$ to the system~\eqref{eqn:gradSystem} given by Proposition \ref{p:duffing}. Then, up to translations and a subsequence, $q_{c_n}\to q_{0}$ in $C^{2}_{loc}(\R,\R^{N})$ where $q_{0}$ is a solution to \eqref{eqn:gradSystem} of homoclinic type at $x_{0}=0$.
\end{prop}
\begin{proof}[{\bf Proof of Proposition~\ref{p:convergenceduffing}}] 
Without loss of generality assume the sequence  $c_{n}\to 0^+$ is so $(c_{n})\subset (0,\cduff)$. Since $c_{n}<\cduff$ we have~\ref{enum:Vc} is satisfied by~\eqref{eq:separationD}, and let us denote 
$q_{n}:=q_{c_{n}}$ the solution given by Proposition \ref{p:duffing}. Recalling~\ref{enum:1cduff}  through~\ref{enum:4cduff}, we denote $(\alpha_{n},\omega_{n}):=(\alpha_{c_{n}},\omega_{c_{n}})$ the
{\sl connecting time interval} of $q_{n}$. Since $\omega_{n}\in\R$ for $n\in\N$, we can assume, up to translations, that $\omega_{n}=0$. That is to say,
\begin{equation}\label{eq:traslduffing}
	\alpha_{n}<0=\omega_{n}\;\;\hbox{ and }\;\;\dot q_{n}(0)=0,\ q_{n}(0)\in \partial\mathcal{V}^{c_{n}}_{+}\,\hbox{ for all }n\in\N.
\end{equation}
Since $\mathcal{C}_{0}$ is bounded and by construction $q_{n}(\R)\subset \mathcal{C}_{0}$ for every $n\in\N$, there exists $\Rduff>0$ such that $\sup_{n\in\N}\|q_{n}\|_{L^{\infty}(\R,\R^{N})}\leq \Rduff$. Since $\ddot q_{n}=\nabla V(q_{n})$ on $\R$, the same arguments in the proof of Proposition~\ref{p:convergence} yield that there exists $q_{0}\in C^{2}(\R,\R^{N})$ such that $\ddot q_{0}=\nabla V(q_{0})$ on $\R$, and $q_{n}\to q_{0}$ in $C_{loc}^{2}(\R,\R^{N})$, along a subsequence, as $n\to +\infty$. The pointwise convergence and \eqref{eq:traslduffing},  together with the fact that $\dist(\partial\mathcal{V}^{c_{n}}_{+},\partial\mathcal{C}_{0})\to 0$ as $n\to+\infty$, imply furthermore that 
\begin{equation}\label{eq:q0C0}
q_{0}(0)\in\partial\mathcal{C}_{0}\;\hbox{ and }\;\dot q_{0}(0)=0.
\end{equation}
Our next goal is to show, similarly to~\eqref{eq:infty}, that
\begin{equation}\label{eq:inftyduffing}
\alpha_{n}\to -\infty\;\hbox{ as }\;n\to+\infty.
\end{equation}
If~\eqref{eq:inftyduffing} were false then we could assume, up to a subsequence, that $\alpha_{n}\to \alpha_{0}\in\R$. By~\ref{enum:2cduff} and \eqref{eq:separationD}, since $\cduff\in (0,\nu_{1})$, then yields $q_{n}(\alpha_{n})\in \mathcal{V}^{c_{n}}_{-}\subset B_{r_0}(0)$. Since $c_{n}\to 0$, it would follow that $q_{n}(\alpha_{n})\to 0$ and, recalling $\dot q_{n}(\alpha_{n})=0$ for any $n\in\N$, we would obtain $q_{0}(\alpha_{0})=0$ and $\dot q_{0}(\alpha_{0})=0$. By uniqueness of solutions to the Cauchy problem, necessarily $q_{0}(t)\equiv 0$ for any $t\in\R$, a contradiction with~\eqref{eq:q0C0}. This shows~\eqref{eq:inftyduffing}. 

To prove the proposition, we are left to show that 
\begin{equation}\label{eq:q0homo}
q_{0}(t)\to 0\;\hbox{ as }\;t\to \pm\infty.
\end{equation}
Note that by \eqref{eq:q0C0}, $q_{0}$ is symmetric with respect to the contact time $\omega_{0}=0$. Hence,~\eqref{eq:q0homo} follows once we show that
\begin{equation}\label{eq:q0homonew}
q_{0}(t)\to 0\;\hbox{ as }\;t\to -\infty.
\end{equation}
The latter reduces to prove that for any $r\in (0,r_0)$, with $r_0$ as in~\ref{enum:duffing:V3}, there exist $L_{r}>0$ and $n_{r}\in\N$ s.t.
\begin{equation}\label{eq:trappingqnduffing} 
|q_{n}(t)|<r\;\;\hbox{ for }\;t\in (\alpha_{n},-L_{r}), \;\text{ for any }\,n\geq n_{r}.
\end{equation}
In order to establish~\eqref{eq:trappingqnduffing} we assume by contradiction that there exist $\bar r\in (0,r_0)$, a subsequence of $(q_{n})$, still denoted $(q_{n})$ and a sequence {$(s_{n})\subset\R$, in such a way that
\begin{equation}\label{eq:snduffing} 
|q_{n}(s_{n})|>\bar r\;\hbox{ with }\;s_{n}\in(\alpha_{n},0),\;\hbox{ and }\; s_{n}\to-\infty.
\end{equation}}
The rest of the proof is devoted to obtain a contradiction with~\eqref{eq:snduffing}, following very similar steps as the proof of {\eqref{eq:trappingqn}  in} Proposition~\ref{p:convergence}. We will only sketch the main ideas and spare some details. The contradiction will be reached by arguing that {if \eqref{eq:snduffing} holds  true, then there exists $(t_{n})$ so that $t_{n}\in (s_{n},0)$ for $n\in\N$ with
\begin{equation}\label{eq:TN}
\lim_{n\to+\infty}\dist(q_n(t_{n}),\partial\mathcal{C}_0\cup\{0\})=0.
\end{equation}
But for such sequence there hold furthermore 
\begin{equation}\label{eq:contr}
\liminf_{n\to+\infty}\dist(q_n(t_{n}),\partial\mathcal{C}_0)>0\;\;\hbox{ and }\;\;\liminf_{n\to+\infty}\dist(q_n(t_{n}),0)>0,
\end{equation}
which is in contradiction with \eqref{eq:TN}. This will prove~\eqref{eq:q0homonew}, completing the proof of Proposition~\ref{p:convergenceduffing}.}\\
To establish \eqref{eq:TN}, we remark that the uniform bound on the energies: $J_{c_{n},(\alpha_{n},0)}(q_{n})=m_{c_{n}}\leq \Mduff$ for all $n\in\N$ (see~Remark~\ref{R:boundmc} since $c_n<\cduff$) together with $s_{n}\to-\infty$ yield that for any $n\in\N$ there exists $t_{n}\in (s_{n},0)$ so that $V(q_n(t_{n}))\to 0$ as $n\to+\infty$. This, in light of~\ref{enum:duffing:V3}-\ref{enum:duffing:V4}, shows  \eqref{eq:TN}.
\smallskip

We next argue \eqref{eq:contr}. If $\liminf_{n\to+\infty}\dist(q_n(t_{n}),\partial\mathcal{C}_0)>0$ fails to hold, then there exists $\xi_0\in\partial\mathcal{C}_0$ so that $q_{n}(t_{n})\to\xi_0$ as $n\to+\infty$, along a subsequence that we continue to denote $(q_n)$. We first note that this behavior of $q_n$ is energetically inexpensive, in that 
\begin{equation}\label{eq:codaduffing}
J_{c_{n},(t_{n},0)}(q_{n})\to 0\;\hbox{ as }\;n\to+\infty.
\end{equation}
This is a consequence of an energy analysis with a suitable competitor $(q_{-,n})$ as follows. By continuity, since $V(q_n(t_{n}))>c_{n}>0=V(\xi_0)$, there is $\sigma_{n}\in (0,1)$ such that
\[
V((1-\sigma_{n})q_{n}(t_{n})+\sigma_{n} \xi_0)=c_{n}\;\;\hbox{ and }\;\;V((1-\sigma)q_{n}(t_{n})+\sigma \xi_0)>c_{n}\,\hbox{
for any }\sigma\in [0,\sigma_{n}).
\]
Hence, the curve
\[
q_{-,n}(t):=\begin{cases} q_{n}(\alpha_{n})&\text{if }\,t<\alpha_{n},\\
q_{n}(t)&\text{if }\,\alpha_{n}\leq t< t_{n},\\
(1-(t-t_{n}))q_{n}(t_{n})+(t-t_{n})\xi_0&\text{if }\,t_{n}\leq t<t_{n}+\sigma_{n},\\
(1-\sigma_{n})q_{n}(t_{n})+\sigma_{n}\xi_0&\text{if }\,t_{n}+\sigma_{n}\leq t.\end{cases}
\]
is an element of $\Gamma_{c_{n}}$, thus $J_{c_{n}}(q_{-,n})\geq m_{c_{n}}$. Since $V((1-\sigma)q_{n}(t_{n})+\sigma \xi_0)\to 0$ as $n\to +\infty$ uniformly for $\sigma\in [0,1]$, we derive that
\[
J_{c_{n},(t_{n},+\infty)}(q_{-,n})=J_{c_{n},(t_{n},t_{n}+\sigma_{n})}(q_{-,n})=\tfrac{\sigma_{n}}2|q_{n}(t_{n})-\xi_0|^{2}+\int_{t_{n}}^{t_{n}+\sigma_{n}}\bigl(V((1-\sigma)q_{n}(t_{n})+\sigma \xi_0)-c_{n}\bigr)d\sigma\to 0.
\]
as $n\to+\infty$. Hence, $J_{c_{n},(\alpha_{n},t_{n})}(q_{n})=J_{c_{n},(\alpha_{n},t_{n})}(q_{-,n})=
J_{c_{n}}(q_{-,n})-J_{c_{n},(t_{n},+\infty)}(q_{-,n})\geq m_{c_{n}}-o(1)$, and since
$m_{c_{n}}=J_{c_{n},(\alpha_{n},t_{n})}(q_{n})+J_{c_{n},(t_{n},0)}(q_{n})$
we conclude~\eqref{eq:codaduffing}.
\smallskip

As an intermediate step, we now claim that
\begin{equation}\label{eq:sup}
\sup_{t\in (t_{n},0)}\dist(q_{n}(t),\partial\mathcal{C}_{0})\to 0\;\text{ as }\;n\to+\infty.
\end{equation} 
Indeed, if not, there is a $\rho_*\in (0,r_{0}/3)$ and a sequence $(\tau_{n})$ with $\tau_{n}\in (t_{n},0)$ for $n\in\N$, such that $\dist(q_{n}(\tau_{n}),\partial\mathcal{C}_{0})=3\rho_*$, up to subsequence. In light of~\eqref{eq:traslduffing}, this implies that there exists $(t_{1,n},t_{2,n})\subset (\tau_{n},0)$ such that
$|q_{n}(t_{2,n})-q_{n}(t_{1,n})|=\rho_*$ and $2\rho_*\geq \dist(q_{n}(t),\partial\mathcal{C}_{0})\geq \rho_*$ for every $t\in (t_{1,n},t_{2,n})$. Hence 
\[
\inf_{t\in (t_{1,n},t_{2,n})}V(q(t))\geq \mu(\rho_*):=\min\{ V(x):x\in \mathcal{C}_{0}\,\hbox{ and }\,2\rho_*\geq \dist(x,\partial\mathcal{C}_{0})\geq\rho_*\}>0,
\] 
and by Remark \ref{rem:disequazioneF} we obtain
\[
J_{c_{n},(t_{n},0)}(q_{n})\geq J_{c_{n},(t_{1,n},t_{2,n})}(q_{n})\sqrt{2(\mu(\rho_*)-c_{n})}|q_{n}(t_{2,n})-q_{n}(t_{1,n})|>\sqrt{\mu(\rho_*)}\tfrac{\rho_*}4, 
\]
for $n$ large enough. This last inequality contradicts~\eqref{eq:codaduffing}, proving~\eqref{eq:sup}. 

\noindent Since $t_{n}\to -\infty$ and $q_{n}\to q_{0}$ in $C^{2}_{loc}(\R,\R^{N})$, by \eqref{eq:sup} we conclude
\begin{equation}\label{eq:q0belongC0}
q_{0}(t)\in\partial\mathcal{C}_{0}\;\hbox{ for any }\;t\leq 0.
\end{equation} 
Also, the energy constraint $E_{q_{n}}(t)=-c_{n}$ for every $t\in\R$, together with the pointwise convergence show that $E_{q_{0}}(t)=0$ for any $t\in\R$. That is to say, $\frac 12|\dot q_{0}(t)|^{2}=V(q_0(t))$ for $t\in\R$ and since $V(x)=0$ for $x\in\partial\mathcal{C}_{0}$, by \eqref{eq:q0belongC0} we obtain $\dot q_{0}(t)=0$ for any $t<0$. Thus, $q_{0}$ is constant with $\ddot q(t)=0$ for $t<0$. Nonetheless, using that $\nabla V\not=0$ on $\partial\mathcal{C}_{0}$, see~\ref{enum:duffing:V4}, we would simultaneously have that $\ddot q_{0}(t)=\nabla V(q_{0}(t))\not=0$ for $t<0$ (by \eqref{eq:q0belongC0}). This contradiction  {proves that}
\[
\liminf_{n\to+\infty}\dist(q_n(t_{n}),\partial\mathcal{C}_0)>0.
\]
To conclude the proof, let us finally show {that $\liminf_{n\to+\infty}\dist(q_n(t_{n}),0)>0$.}
Assume by contradiction that $q_{n}(t_{n})\to 0$ and, as in the two-well case, for any $n\in\N$ we fix $\sigma_{n}\in (0,1)$ such that
\[
V((1-\sigma_{n}) 0+\sigma_{n} q_{n}(t_{n}))=c_{n}\;\;\hbox{ and }\;\;V((1-\sigma) 0+\sigma q_{n}(t_{n}))>c_{n}\,\hbox{
for any }\sigma\in (\sigma_{n},1].
\]
Then  
\[q_{+,n}(t):=\begin{cases} \sigma_{n} q_{n}(t_{n})&\text{if }\, t<t_{n}-(1-\sigma_{n}),\\
(t-t_{n}+1) q_{n}(t_{n})&\text{if }\,t_{n}-(1-\sigma_{n})\leq t<t_{n},\\
q_{n}(t)&\text{if }\,t_{n}\leq t<0,\\
q_{n}(0)&\text{if }\, 0\leq t.
\end{cases}\]
is in $\Gamma_{c_{n}}$, whence $J_{c_{n}}(q_{+,n})\geq m_{c_{n}}$, and just as in the two-well case, we obtain as $n\to+\infty$,
\begin{equation}\label{eq:coda2duffing}
J_{c_{n},(\alpha_{n},t_{n})}(q_{n})\to 0\;\;\hbox{ as }\;\;n\to+\infty.\end{equation}
In summary, we are in a situation where
\[
\alpha_{n}<s_{n}<t_{n}<0,\; |q_{n}(s_{n})|>\bar r\;\hbox{ and }\;q_{n}(t_{n})\to 0\; \text{ as }\;n\to+\infty.
\]
Then, when $n$ is large, there exists $(t_{3,n},t_{4,n})\subset (s_{n},t_{n})$ such that
$|q_{n}(t_{4,n})-q_{n}(t_{3,n})|=\tfrac{\bar r}4$ and $|q_{n}(t)|\geq \tfrac{\bar r}4$ for any $t\in (t_{3,n},t_{4,n})$. Letting $\bar\mu(\bar r):=\min\{V(x):\tfrac{\bar r}4\leq|x|\leq \bar r\}$ we can apply Remark~\ref{rem:disequazioneF} similarly as above, to obtain for $n$ sufficiently large: 
\[
J_{c_{n},(\alpha_{n},t_{n})}(q_{n})\geq J_{c_{n},(t_{3,n},t_{4,n})}(q_{n})>\sqrt{\bar\mu(\bar r)}\tfrac{\bar r}4,
\]
This contradicts~\eqref{eq:coda2duffing}, which proves that $q_n(t_n)\to 0$ is not possible either; {$\liminf_{n\to+\infty}\dist(q_n(t_{n}),0)>0$} is now established.
\end{proof}

\subsection{The multiple pendulum type systems.} As a last classical example, we consider the case of {\em multiple pendulum type} systems. That is, we assume (see e.g. \cite{[ABM]}, \cite{[BM]},\cite{rabinowitz2000results}, for analogous assumptions) 
\begin{enumerate}[label={\rm(V\arabic*)},start=5,itemsep=2pt]
 \item\label{enum:pendulum:V5} $V\in C^{1}(\R^N)$ is ${\Z}^{N}$-periodic,
 \item\label{enum:pendulum:V6} $V(x)\geq 0$, and $V(x)=0$ if and only if $x\in {\Z}^{N}$. 
 \end{enumerate}
Upon assuming~\ref{enum:pendulum:V5} and~\ref{enum:pendulum:V6} we observe that $\mathcal{V}^{c}=\Z^{N}$ for $c=0$. Thus, by continuity and periodicity there exists $c_{p}>0$ so that
\[
\mathcal{V}^{c}\subset  
{\textstyle\bigcup\limits_{\xi\in\Z^{n}}}B_{1/3}(\xi)\;\hbox{ for any }\,c\in [0,c_{p}).
\] 
Hence, by denoting for any such $c\in [0,c_{p})$ 
\[\mathcal{V}^{c}_{\xi}:=\mathcal{V}^{c}\cap B_{1/3}(\xi)\;\hbox{ for }\;\xi\in\Z^{N},\]
we observe the following properties hold:
\begin{enumerate}[label=($\mathcal{V}$\roman*),start=1,itemsep=2pt]
\item\label{enum:pendulum:Vi} $\mathcal{V}^{c}={\textstyle\bigcup_{\xi\in\Z^{N}}}\mathcal{V}^{c}_{\xi}$,
\item\label{enum:pendulum:Vii} $\mathcal{V}^{c}_{\xi}$ is compact for any $\xi\in\Z^{N}$,
\item\label{enum:pendulum:Viii} $\exists r_{c}>0$ such that $\dist(\mathcal{V}^{c}_{\xi},\mathcal{V}^{c}_{\xi'})\geq 2r_{c}$
for any pair $\xi\not=\xi'\in\Z^{N}$,
\item\label{enum:pendulum:Viv} $\forall r\in (0,r_{c})$, $\exists \mu_{r}>0$ such that
$V(x)>c+\mu_{r}$ for any $x\in \R^{N}\setminus {\textstyle\bigcup_{\xi\in\Z^{N}}}B_{r}(\mathcal{V}^{c}_{\xi})\, $.
\end{enumerate}
 
\begin{Remark}\label{rem:test}
Let us denote the elements of the canonical basis of $\R^{N}$ by {${\bm e}_\ell$}, for $\ell=1,\ldots,N$, and define associated functions ${\zeta}_{\ell}:\R\to\R^N$ by $\zeta_{\ell}(t)=t{\bm e}_\ell$ for $t\in [0,1]$, $\zeta_{\ell}(t)=0$ for $t\leq 0$ and $\zeta_{\ell}(t)={\bm e}_\ell$ for $t\geq 1$. Since by definition $\mathcal{V}^{c}_{\xi}\subset B_{1/3}(\xi)$ for any $\xi\in\Z^{N}$ and $c\in [0,c_{p})$, elementary geometric considerations give that for any $\ell\in\{1,\ldots, N\}$ and $c\in [0,c_{p})$ there exists $(s_{c,\ell},t_{c,\ell})\subset [0,1]$ in such a way that
\[
V(\zeta_{\ell}(t))>c\;\text{ for any }\,t\in (s_{c,\ell},t_{c,\ell}),\; \zeta_{\ell}(s_{c,\ell})\in\partial \mathcal{V}^{c}_{0}\;\text{ and }\;\zeta_{\ell}(t_{c,\ell})\in\partial \mathcal{V}^{c}_{{\bm e}_{\ell}}.
\]
Next, it will be convenient to introduce the following test function $\eta_{c,\ell}:\R\to\R^N$ given by $\eta_{c,\ell}(t)=\zeta_{\ell}(t)$ for $t\in (s_{c,\ell},t_{c,\ell})$, 
${\eta_{c,\ell}}(t)=\zeta_{\ell}(s_{c,\ell})$ for $t\leq s_{c,\ell}$ and ${\eta_{c,\ell}}(t)=\zeta_{\ell}(t_{c,\ell})$ for $t\geq t_{c,\ell}$. {We readily get the bound} 
\begin{equation}\label{eq:boundmcsecondo}
J_{c}(\eta_{c,\ell})\leq J_{0}(\zeta_{\ell})=\int_{0}^{1}
\tfrac 12|{\bm e}_\ell|^{2}+V(t{\bm e}_\ell)\, dt\leq {M_{p}}:=\tfrac 12+\max_{|x|\leq 1} V(x),
\end{equation}
for any $\ell\in\{1,\ldots,N\}$ and $c\in[0,c_{p})$ as above.
\end{Remark}

We first observe 
\begin{Lemma}\label{L:coercPend}  
There exists $R_{p}>0$ so that {for any $c\in [0,c_{p})$}, if $q\in W^{1,2}_{loc}(\R,\R^{N})$ and $(\sigma,\tau)\subset\R$ are such that
$V(q(t))\geq c$ for $t\in (\sigma,\tau)$ and $J_{c,(\sigma,\tau)}(q)\leq M_{p}+1$, then 
\[
|q(\tau)-q(\sigma)|\leq R_{p}.
\]
\end{Lemma}
\begin{proof}[{\bf Proof of Lemma~\ref{L:coercPend}}]
{Let us write $r_{p}:=\frac12r_{c_{p}}$, and $\mu_{p}:=\mu_{r_{p}}$ as in~\ref{enum:pendulum:Viii}-\ref{enum:pendulum:Viv}.} Then it follows that {$V(x)>c_{p}+{\mu_{p}}$} for any $x\in \R^{N}\setminus {{\textstyle\bigcup_{\xi\in\Z^{N}}}B_{r_{p}}(\mathcal{V}^{c_{p}}_{\xi})}$. Now, to establish Lemma~\ref{L:coercPend} let us assume by contradiction that there are sequences $(c_n)\subset [0,c_{p})$ and $(q_{n})\subset W^{1,2}_{loc}(\R,\R^{N})$ with corresponding intervals $(\sigma_{n},\tau_{n})\subset\R$ such that $V(q_{n}(t))\geq c_{n}$ for any $t\in (\sigma_{n},\tau_{n})$, $J_{{c_n},(\sigma_{n},\tau_{n})}(q_{n})\leq M_{p}+1$, and
\[
|q_{n}(\tau_{n})-q_{n}(\sigma_{n})|\geq 2n\sqrt{N}.
\]
In particular, since {$c_n\leq c_{p}$} we see from~\ref{enum:pendulum:Viii} that $\dist(\mathcal{V}^{c_n}_{\xi},\mathcal{V}^{c_n}_{\xi'}){\geq 2r_{c_{p}}}$ for all $n\in\N$, if $\xi\neq\xi'$. These inequalities, along with basic geometric considerations, cf.~\ref{enum:pendulum:V5}, imply the existence of $n$ disjoint intervals $(s_i,t_i)\subset(\sigma_{n}, \tau_n)$ for $1\leq i\leq n$, such that $q_n(t)\notin {{\textstyle\bigcup_{\xi\in\Z^{N}}}B_{r_{p}}(\mathcal{V}^{c_{p}}_{\xi})}$ if $t\in{\textstyle\bigcup_{i}}(s_i,t_i)$, while $|q_n(t_i)-q_n(s_i)|\geq {2(r_{c_{p}}-r_{p})=r_{c_{p}}}$. Then, by Remark \ref{rem:disequazioneF}, $J_{c_n,(\sigma_n,\tau_n)}(q_n)\geq {n\sqrt{2\mu_{p}}r_{c_{p}}}$
for any $n\in\N$. But this goes in contradiction with $J_{{c_n},(\sigma_{n},\tau_{n})}(q_{n})\leq M_{p}+1$,  and the Lemma follows.
\end{proof}

\smallskip

The above properties allow us to apply Theorem~\ref{P:main1}, giving the next
\begin{prop}\label{p:pendumum} 
Assume that $V$ satisfies~\ref{enum:pendulum:V5} and~\ref{enum:pendulum:V6}. Then for every $c\in [0,{c_{p}})$ there exists $k_c\in\N$ and a finite set $\{\xi^{1},\ldots,{\xi^{k_c}}\}\subset \Z^{N}\setminus\{0\}$, satisfying
\begin{equation}\label{eq:generate}
\xi^i\neq\xi^j\;\text{ for }\;i\neq j,\;\hbox{ and }\;\{ n_{1}\xi^{1}+\ldots+n_{k_c}\xi^{k_c}:\, n_{1},\ldots, n_{k_c}\in\Z\}=\Z^{N},
\end{equation}
for which, given any $j\in\{1,\ldots,k_c\}$, there is a solution $q_{c,j}\in C^{2}(\R,\R^N)$ to~\eqref{eqn:gradSystem} with energy {$E_{q_{c,j}}(t)=-c$ for any $t\in\R$}, verifying
\begin{equation}\label{eqn:cond-pend}
\inf_{t\in\R}\dist(q_{c,j}(t),\mathcal{V}^{c}_0)=\inf_{t\in\R}\dist(q_{c,j}(t),\mathcal{V}^{c}_{\xi^{j}})=0\;\text{ and }\;\|q_{c,j}\|_{L^{\infty}(\R,\R^{N})}\leq R_{p}+1,
\end{equation}
where $R_{p}$ is given by Lemma \ref{L:coercPend}.
\end{prop}

\begin{proof}[{\bf Proof of Proposition~\ref{p:pendumum}}] For any $c\in [0,{c_{p}})$, let us set
\begin{equation}\label{eq:decompP1}
\mathcal{V}^{c}_{-,1}:=\mathcal{V}^{c}_{0}\;\;\hbox{ and }\;\;\mathcal{V}^{c}_{+,1}:={\textstyle\bigcup\limits_{\xi\in\Z^{N}\setminus\{0\}}}\mathcal{V}^{c}_{\xi}=\mathcal{V}^{c}\setminus\mathcal{V}^{c}_{-,1}.
\end{equation}
In particular, we observe in light of~\ref{enum:pendulum:Viii} that
\[
\dist(\mathcal{V}^{c}_{-,1},\mathcal{V}^{c}_{+,1})\geq 2r_{c},
\]
so~\ref{enum:Vc} holds. Then defining $\Mm_{c,1}$ as in~\eqref{defMc}, relative to the partition~\eqref{eq:decompP1}, we let
\begin{equation*}\label{eq:m1}
m_{c,1}:=\inf_{q\in\Gamma_{c,1}}{J_{c}}(q).
\end{equation*}
Recall the definition of $\eta_{c,\ell}$ in Remark \ref{rem:test} and observe that $\eta_{c,1}\in \Gamma_{c,1}$. {This, together with~\eqref{eq:boundmcsecondo},} yields
\[
m_{c,1}\leq J_{c}(\eta_{c,1})\leq M_{p}.
\]
Consider now any minimizing sequence $(q_{n})\subset\Mm_{c,1}$, so that $J_{c}(q_{n})\to m_{c,1}$. Eventually passing to a subsequence, we can assume that {$J_{c}(q_{n})\le  M_{p}+1$} for any $n\in\N$, and so by Lemma~\ref{L:coercPend}
\begin{equation*}\label{eq:claimP1}
|q_{n}(t)-q_{n}(s)|\leq R_{p}\;\text{ for any }(s,t)\subset\R,\text{ and }\,n\in\N.
\end{equation*}
Therefore the coercivity condition~\eqref{eqn:coercivita} of $J_c$ is satisfied for the division~\eqref{eq:decompP1}, hence Theorem~\ref{P:main1} gives the existence of a solution $q_{c,1}\in C^{2}(\R,\R^N)$ to~\eqref{eqn:gradSystem} with energy
$E_{q_{c,1}}(t)=-c$ for all $t\in\R$, satisfying
\[
\liminf_{t\to -\infty}\dist(q_{c,1}(t),\mathcal{V}^{c}_{-,1})=0\;\;\text{ and }\;\;\liminf_{t\to +\infty}\dist(q_{c,1}(t), \mathcal{V}^{c}_{+,1})=0.
\]
This solution is either of brake orbit type, case~\ref{enum:main:1} of the theorem, of homoclinic type, case~\ref{enum:main:2}, or of heteroclinic type in the case~\ref{enum:main:3}. We now continue with the proof of Proposition~\ref{p:pendumum} by checking that~\eqref{eqn:cond-pend} is satisfied, regardless of the case in consideration.

If we are in the case (a) there exists $-\infty<\sigma<\tau<+\infty$ such that  
\begin{enumerate}
\item[$(a_0)$]\label{enum:qi} $q_{c,1}(\sigma)\in\mathcal{V}^{c}_{-,1}$, $q_{c,1}(\tau)\in\mathcal{V}^{c}_{+,1}$,  $V(q_{c,1}(t))>c$ for $t\in (\sigma,\tau)$, $q_{c,1}(\sigma+t)=q_{c,1}(\sigma-t)$ and $q_{c,1}(\tau+t)=q_{c,1}(\tau-t)$ for all $t\in\R$.
\end{enumerate}
Let us now point out that $(a_0)$, \eqref{eq:decompP1} and~\ref{enum:pendulum:Vi} yield that $q_{c,1}(\sigma)\in\mathcal{V}^{c}_{0}$ and that there exists $\xi^{1}\in\Z^{N}\setminus\{0\}$ with $q_{c,1}(\tau)\in\mathcal{V}^{c}_{\xi^{1}}$. Moreover, from Remark~\ref{R:min} we have that for any $(s,t)\subset (\sigma,\tau)$, $J_{c,(s,t)}(q_{c,1})\leq J_{c,(\sigma,\tau)}(q_{c,1})=m_{c,1}\leq  M_{p}$ and so Lemma~\ref{L:coercPend} gives in particular $|q_{c,1}(\sigma)-q_{c,1}(t)|\leq R_{p}$ for any $t\in (\sigma,\tau)$. By periodicity we then obtain $\dist(q_{c,1}(t),\mathcal{V}^{c}_{0})\leq R_{p}$ for any $t\in\R$. Since $\mathcal{V}^{c}_{0}\subset B_{1/3}(0)$ we get $\|q_{c,1}\|_{L^{\infty}(\R,\R^{N})}\leq R_{p}+1$, therefore property~\eqref{eqn:cond-pend} is satisfied by $q_{c,1}$.

If we are in the case (b), then there exists $\sigma\in\R$ such that  
\begin{enumerate}
\item[$(b_0)$]\label{enum:qhi} $q_{c,1}(\sigma)\in\mathcal{V}^{c}_{\pm,1}$, $\lim_{t\to\pm\infty}\dist(q_{c,1}(t),\mathcal{V}^{c}_{\mp,1})=0$,
 $V(q_{c,1}(t))>c$ for $t\in \R\setminus\{\sigma\}$, and
 $q_{c,1}(\sigma+t)=q_{c,1}(\sigma-t)$ for all $t\in\R$.
\end{enumerate}
Once again, we point out that $(b_0)$, \eqref{eq:decompP1} and~\ref{enum:pendulum:Vi} imply that there is $\xi^{1}\in\Z^{N}\setminus\{0\}$ such that either $q_{c,1}(\sigma)\in\mathcal{V}^{c}_{0}$ and $\lim_{t\to\pm\infty}\dist(q_{c,1}(t),\mathcal{V}^{c}_{\xi^{1}})=0$, or $q_{c,1}(\sigma)\in\mathcal{V}^{c}_{\xi^{1}}$ and  $\lim_{t\to\pm\infty}\dist(q_{c,1}(t),\mathcal{V}^{c}_{0})=0$. 
By Remark~\ref{R:min} we have $J_{c,(s,t)}(q_{c,1})\leq J_{c,(-\infty,\sigma)}(q_{c,1})=m_{c,1}\leq  M_{p}$ for any $(s,t)\subset (-\infty,\sigma)$. This, combined with Lemma~\ref{L:coercPend} and the reflection $q_{c,1}(\sigma+t)=q_{c,1}(\sigma-t)$ for all $t\in\R$, shows that $|q_{c,1}(t)-q_{c,1}(s)|\leq R_{p}$ for any $s<t\in \R$. But since $q_{c,1}(\sigma)\in\mathcal{V}^{c}_{\pm,1}$ and $\lim_{t\to\pm\infty}\dist(q_{c,1}(t),\mathcal{V}^{c}_{\mp,1})=0$ we obtain
$\dist(q_{c,1}(t),\mathcal{V}^{c}_{0})\leq R_{p}$ for any $t\in\R$ and so  $\|q_{c,1}\|_{L^{\infty}(\R,\R^{N})}\leq R_{p}+1$. This shows again \eqref{eqn:cond-pend} for $q_{c,1}$.

Finally if we are in the case (c) we have 
\begin{enumerate}
\item[$(c_0)$]\label{enum:qhhi}  $V(q_{c,1}(t))>c$ for all $t\in\R$, $\lim\limits_{t\to-\infty}\dist(q_{c,1}(t),\mathcal{V}^{c}_{-,1})=0\,$ and $\lim\limits_{t\to+\infty}\dist(q_{c,1}(t),\mathcal{V}^{c}_{+,1})=0$.\end{enumerate}
As before $(c_0)$, \eqref{eq:decompP1} and \ref{enum:pendulum:Vi} show that there is $\xi^{1}\in\Z^{N}\setminus\{0\}$ such that $\lim\limits_{t\to+\infty}\dist(q_{c,1}(t),\mathcal{V}^{c}_{\xi^{1}})=0$. From Remark~\ref{R:min} we have  $J_{c,(s,t)}(q_{c,1})\leq J_{c}(q_{c,1})=m_{c,1}\leq M_{p}$ for any $(s,t)\subset \R$, which combined with Lemma~\ref{L:coercPend} gives $|q_{c,1}(t)-q_{c,1}(s)|\leq R_{p}$ for any $s<t\in \R$. Since $\lim_{t\to-\infty}\dist(q_{c,1}(t),\mathcal{V}^{c}_{0})=0$, we deduce that $\dist(q_{c,1}(s),\mathcal{V}^{c}_{0})\leq R_{p}$ for any $s\in\R$ and hence $\|q_{c,1}\|_{L^{\infty}(\R,\R^{N})}\leq R_{p}+1$.  Whence, even in case (c) the condition~\eqref{eqn:cond-pend} holds for $q_{c,1}$.\\
The above argument shows the existence of a solution $q_{c,1}$ with energy $E_{q_{c,1}}=-c$ with
\[
\inf_{t\in\R}\dist(q_{c,1}(t),\mathcal{V}^{c}_0)=\inf_{t\in\R}\dist(q_{c,1}(t),\mathcal{V}^{c}_{\xi^{1}})=0\;\,\text{ and }\;\,\|q_{c,1}\|_{L^{\infty}(\R,\R^{N})}\leq R_{p}+1.
\]
In order to establish the multiplicity of solutions in Proposition~\ref{p:pendumum}, we will proceed by induction. 
Let us assume that for $j\geq 1$ we have
\begin{enumerate}[label=(\Roman*),start=1,itemsep=2pt]
\item\label{enum:ind1} There are $\xi^{1},\xi^{2},..., \xi^{j}\in\Z^{N}\setminus\{0\}$ such that $\xi^h\ne\xi^k$ for $h\ne k$ and if we set $\mathcal{L}_{j}:=\{\sum_{i=1}^{j}n_{i}\xi^{i}:\, n_{1},\ldots,n_{j}\in\Z\}$, then $\mathcal{L}_{j}\not=\Z^{N}.$
\item\label{enum:ind2} For any $i\in\{1,\ldots,j\}$ there exists $q_{c,i}$ with energy $E_{q_{c,i}}=-c$ such that 
\begin{equation*}\label{eq:connect}
\inf_{t\in\R}\dist(q_{c,i}(t),\mathcal{V}^{c}_0)=\inf_{t\in\R}\dist(q_{c,i}(t),\mathcal{V}^{c}_{\xi^{i}})=0\;\text{ and }\;\|q_{c,i}\|_{L^{\infty}(\R,\R^{N})}\leq R_{p}+1.
\end{equation*}
\end{enumerate} 
Proposition~\ref{p:pendumum} will follow once we show that \ref{enum:ind1} and \ref{enum:ind2} together imply the existence of $\xi^{j+1}\in\Z^{N}\setminus\mathcal{L}_{j}$ and a solution $q_{c,j+1}$ with energy $E_{q_{c,j+1}}=-c$, in such a way that
\begin{equation}\label{eq:connecttesi}
\inf_{t\in\R}\dist(q_{c,j+1}(t),\mathcal{V}^{c}_0)=\inf_{t\in\R}\dist(q_{c,j+1}(t),\mathcal{V}^{c}_{\xi^{j+1}})=0\;\text{ and }\;\|q_{c,j+1}\|_{L^{\infty}(\R,\R^{N})}\leq R_{p}+1.
\end{equation}
To see this, we consider in view of~\ref{enum:ind1}, the following decomposition of $\mathcal{V}^{c}$: 
\begin{equation}\label{eq:decompPj}
\mathcal{V}^{c}_{-,j+1}:={\textstyle\bigcup_{\xi\in\mathcal{L}_{j}}}\mathcal{V}^{c}_{\xi}\;\;\hbox{ and }\;\;\mathcal{V}^{c}_{+,j+1}:=\mathcal{V}^{c}\setminus\mathcal{V}^{c}_{-,j+1}={\textstyle\bigcup_{\xi\in\Z^{N}\setminus\mathcal{L}_{j}} }\mathcal{V}^{c}_{\xi}.\end{equation}
In light of~\ref{enum:ind1} both sets $\mathcal{V}^{c}_{-,j+1}$, $\mathcal{V}^{c}_{+,j+1}$ are non-empty, they clearly verify $\dist(\mathcal{V}^{c}_{-,j+1},\mathcal{V}^{c}_{+,j+1})\geq r_{c},$ and 
\begin{equation}\label{enum:Vcj2}
\mathcal{V}^{c}_{\pm,j+1}=\xi+\mathcal{V}^{c}_{\pm,j+1}\,\hbox{ for any }\,\xi\in\mathcal{L}_{j},
\end{equation}
see Figure \ref{fig:multiple} below.
\begin{figure}[!h]
\def\svgwidth{1\columnwidth}
\begin{center} 
\input{pendoloheteroFv2.pdf_tex}
\end{center}
 \caption{\small Possible configurations in pendulum like systems.}\label{fig:multiple}
\end{figure} 

Then, according to the partition~\eqref{eq:decompPj} of $\mathcal{V}^{c}$,
define $\Mm_{c,j+1}$ as in \eqref{defMc} and let $m_{c,j+1}:=\inf_{q\in\Gamma_{c,j+1}}{J_{c}}(q)$. Since $\mathcal{L}_{j}\not=\Z^{N}$ there must exist $\ell_{j+1}\in\{1,\ldots, N\}$ so that ${\bm e}_{\ell_{j+1}}\notin\mathcal{L}_{j}$. Then, {by Remark \ref{rem:test},} $\eta_{c,\ell_{j+1}}\in \Mm_{c,j+1}$, and so by \eqref{eq:boundmcsecondo}
$$m_{c,j+1}\leq J_{c}(\eta_{c,\ell_{j+1}})\leq  M_{p}.$$
Let $(q_{n})\subset\Mm_{c,j+1}$ be such that $J_{c}(q_{n})\to m_{c,j+1}$. With no loss of generality we can assume that $J_{c}(q_{n})< M_{p}+1$ for any $n\in\N$, and so by Lemma~\ref{L:coercPend} we obtain
\begin{equation}\label{eq:boundj}
|q_{n}(t)-q_{n}(s)|\leq R_{p}\;\text{ for any }\,(s,t)\subset\R,\,\text{ and }\,n\in\N.
\end{equation}
Since $q_{n}\in\Mm_{c,j+1}$ we have $\liminf_{t\to-\infty}\dist(q_{n}(t),\mathcal{V}^{c}_{-,j+1})=0$ for any $n\in\N$. Hence, there exist sequences $(s_{n})\subset\R$ and $(\zeta_{n})\subset \mathcal{L}_j$ for which
\begin{equation}\label{eq:bddLattice1}
|q_{n}(s_{n})-\zeta_{n}|<1\;\text{ for all }\, n\in\N.
\end{equation}
By periodicity of $V$, \eqref{enum:Vcj2} and by \eqref{eq:boundj}-\eqref{eq:bddLattice1}, we have 
\begin{equation*}\label{eq:claimPj} 
\tilde q_{n}:=q_{n}-\zeta_{n}\in\Mm_{c,j+1},\ J_{c}(\tilde q_{n})\to m_{c,j+1}\;\text{ and }\;
\| \tilde q_{n}\|_{L^{\infty}(\R,\R^{N})}\leq R_{p}+1\;\hbox{ for any }n\in\N,
\end{equation*}
from which we conclude that
\begin{equation}\label{eq:importantBdd}
m_{c,j+1}= \inf\{ J_{c}(q): q\in\Gamma_{c,j+1},\,\|q\|_{L^{\infty}(\R,\R^N)}\leq R_{p}+1\},
\end{equation}
namely, the coercivity condition~\eqref{eqn:coercivita} of $J_c$ over $\Gamma_{c,j+1}$ follows. Thus, Theorem~\ref{P:main1} yields the existence of a solution $ \bar q_{c,j+1}\in C^{2}(\R,\R^N)$ to~\eqref{eqn:gradSystem} with energy $E_{ \bar q_{c,j+1}}=-c$, satisfying
\[
\liminf_{t\to -\infty}\dist(\bar q_{c,j+1}(t), \mathcal{V}^{c}_{-,j+1})=0\;\;\text{ and }\;\;\liminf_{t\to +\infty}\dist( \bar q_{c,j+1}(t), \mathcal{V}^{c}_{+,j+1})=0.
\]
This solution is either of brake orbit type, case~\ref{enum:main:1}, of homoclinic type, case~\ref{enum:main:2}, or of heteroclinic type in the case~\ref{enum:main:3}. 
If we are in the case~\ref{enum:main:1}, then there exists $-\infty<\sigma<\tau<+\infty$ such that  
\begin{enumerate}
\item[$(a_j)$]\label{enum:qhiv} $\bar q_{c,j+1}(\sigma)\in\mathcal{V}^{c}_{-,j+1}$, $\bar q_{c,j+1}(\tau)\in\mathcal{V}^{c}_{+,j+1}$,  $V( \bar q_{c,j+1}(t))>c$ for $t\in (\sigma,\tau)$, $\bar q_{c,j+1}(\sigma+t)=\bar q_{c,j+1}(\sigma-t)$ and $ \bar q_{c,j+1}(\tau+t)= \bar q_{c,j+1}(\tau-t)$ for all $t\in\R$.
\end{enumerate}
By $(a_j)$ there exists $\xi_{-}\in \mathcal{L}_{j}$, $\xi_{+}\in\Z^{N}\setminus\mathcal{L}_{j}$ such that $\bar q_{c,j+1}(\sigma)\in\mathcal{V}^{c}_{\xi_{-}}$ and $\bar q_{c,j+1}(\tau)\in\mathcal{V}^{c}_{\xi_{+}}$. Then $\xi^{j+1}=\xi_{+}-\xi_{-}\in\Z^{N}\setminus\mathcal{L}_{j}$, and by periodicity of $V$, the function $q_{c,j+1}:=\bar q_{c,j+1}-\xi_{-}$ is a solution to~\eqref{eqn:gradSystem} with energy $E_{q_{c,j+1}}=-c$. By Remark~\ref{R:min}
we have moreover $J_{c,(s,t)}(q_{c,j+1})\leq J_{c,(\sigma,\tau)}(q_{c,j+1})=m_{c,j+1}\leq  M_{p}$ for any $(s,t)\subset (\sigma,\tau)$, and so, arguing as in the case $(a_0)$ above we deduce $\|q_{c,j+1}\|_{L^{\infty}(\R,\R^{N})}\leq R_{p}+1$. Then property~\eqref{eq:connecttesi} with respect to $\xi^{j+1}$ is satisfied by $q_{c,j+1}$.\\
In case~\ref{enum:main:2} there exists $\sigma\in\R$ such that  
\begin{enumerate}
\item[$(b_j)$]\label{enum:qhv} $\bar q_{c,j+1}(\sigma)\in\mathcal{V}^{c}_{\pm,j+1}$, $\lim_{t\to\pm\infty}\dist(\bar q_{c,j+1}(t),\mathcal{V}^{c}_{\mp,j+1})=0$,
 $V(\bar q_{c,j+1}(t))>c$ for $t\in \R\setminus\{\sigma\}$, and
 $\bar q_{c,j+1}(\sigma+t)=\bar q_{c,j+1}(\sigma-t)$ for all $t\in\R$.
\end{enumerate}
In particular, from $(b_j)$ we deduce the existence of $\bar\xi\in \Z^{N}$ so that $\bar q_{c,j+1}(\sigma)\in\mathcal{V}^{c}_{\bar\xi}$. Let us say that $\bar\xi\in \mathcal{L}_{j}$, then $\lim_{t\to\pm\infty}\dist(\bar q_{c,j+1}(t),\mathcal{V}^{c}_{+,j+1})=0$. The symmetry of $\bar q_{c,j+1}$ with respect to $\sigma$ together with the discreteness of $\mathcal{V}^{c}_{+,j+1}$ implies that there is $\xi_{\infty}\in \Z^{N}\setminus \mathcal{L}_{j}$ so that
$\lim_{t\to\pm\infty}\dist(\bar q_{c,j+1}(t),\mathcal{V}^{c}_{\xi_{\infty}})=0$. In this case we set $\xi^{j+1}:=\xi_{\infty}-\bar\xi\in\Z^{N}\setminus\mathcal{L}_{j}$ and the function $q_{c,j+1}:=\bar q_{c,j+1}-\bar\xi$ is a solution to~\eqref{eqn:gradSystem} with energy
$E_{\bar q_{c,j+1}}=-c$, which verifies the first part of~\eqref{eq:connecttesi} with respect to $\xi^{j+1}$.  Otherwise, if $\bar\xi\in\Z^{N}\setminus\mathcal{L}_{j}$, then $\lim_{t\to\pm\infty}\dist(\bar q_{c,j+1}(t),\mathcal{V}^{c}_{-,j+1})=0$. The symmetry of  $\bar q_{c,j+1}$ with respect to $\sigma$ and the discreteness of $\mathcal{V}^{c}_{-,j+1}$ imply the existence of $\xi_{\infty}\in \mathcal{L}_{j}$ so that $\lim_{t\to\pm\infty}\dist(\bar q_{c,j+1}(t),\mathcal{V}^{c}_{\xi_{\infty}})=0$.
In this case we let $\xi^{j+1}:=\bar\xi-\xi_{\infty}\in\Z^{N}\setminus\mathcal{L}_{j}$, and {again} the function $q_{c,j+1}:=\bar q_{c,j+1}-\xi_{\infty}$ is a solution to~\eqref{eqn:gradSystem} with energy $E_{q_{c,j+1}}=-c$ which verifies the first part of~\eqref{eq:connecttesi} with respect to $\xi^{j+1}$.\\
To get the second part of \eqref{eq:connecttesi} observe that by Remark~\ref{R:min}
we have $J_{c,(s,t)}(q_{c,j+1})\leq J_{c,(-\infty,\sigma)}(q_{c,j+1})=m_{c,j+1}\leq  M_{p}$ for any $(s,t)\subset (-\infty,\sigma)$. Then Lemma~\ref{L:coercPend} and the symmetry property of $q_{c,j+1}$ with respect to $\sigma$ allow us to argue as in the case $(b_0)$ above to deduce $\|q_{c,j+1}\|_{L^{\infty}(\R,\R^{N})}\leq R_{p}+1$, thus showing \eqref{eq:connecttesi}.

Finally, let us assume that we are in case~\ref{enum:main:3}. Then
\begin{enumerate}
\item[$(c_j)$]\label{enum:qhvi}  $V(\bar q_{c,j+1}(t))>c$ for all $t\in\R$, and $\lim\limits_{t\to\pm\infty}\dist(\bar q_{c,j+1}(t),\mathcal{V}^{c}_{\pm,j+1})=0$.\end{enumerate}
By invoking once again the discreteness of the sets $\mathcal{V}^{c}_{\pm}$, $(c_j)$ shows the existence of $\xi_{-}\in \mathcal{L}_{j}$ and $\xi_{+}\in\Z^{N}\setminus\mathcal{L}_{j}$ in such a way that $\lim\limits_{t\to\pm\infty}\dist(\bar q_{c,j+1}(t),\mathcal{V}^{c}_{\xi_{\pm}})=0$. Setting $\xi^{j+1}:=\xi_{+}-\xi_{-}\in\Z^{N}\setminus\mathcal{L}_{j}$ we see from the periodicity of $V$ that the function $q_{c,j+1}:=\bar q_{c,j+1}-\xi_{-}$ is a solution to~\eqref{eqn:gradSystem} with energy $E_{q_{c,j+1}}=-c$, verifying the first part of~\eqref{eq:connecttesi} with respect to $\xi^{j+1}$. By Remark~\ref{R:min}
we have  $J_{c,(s,t)}(q_{c,j+1})\leq J_{c}(q_{c,j+1})=m_{c,j+1}\leq M_{p}$ for any $(s,t)\subset \R$. Using Lemma~\ref{L:coercPend} and arguing as in the case {$(c_0)$} above we obtain again $\|q_{c,j+1}\|_{L^{\infty}(\R,\R^{N})}\leq R_{p}+1$.  Then 
\eqref{eq:connecttesi} follows for $q_{c,j+1}$.\smallskip

This concludes the proof of the inductive step, and {hence the proof of Proposition~\ref{p:pendumum}}. 
\end{proof}

\begin{Remark}\label{R:pendolumsol} 
Proposition~\ref{p:pendumum} constitutes a multiplicity result. It asserts the existence of $k_{c}$ elements $\xi^1,\ldots, \xi^{k_{c}}$ in the lattice $\Z^{N}\setminus\{0\}$, for each of which there exists a connecting orbit between $\mathcal{V}^{c}_{0}$ and $\mathcal{V}^{c}_{\xi^{j}}$. In particular, we necessarily have $k_{c}\geq N$, due to~\eqref{eq:generate}. 
\end{Remark}

As for the preceding cases, we finalize by analyzing the convergence properties of the family of {solutions $q_{c,j}$} given by Proposition~\ref{p:pendumum} as $c\to 0^{+}$.

\begin{prop}\label{p:convergencepend} Assume that $V\in C^{1}(\R^{N})$ satisfies~\ref{enum:pendulum:V5}-\ref{enum:pendulum:V6} and that $\nabla V$ is locally Lipschitz continuous in $\R^{N}$. For any sequence $c_{n}\to 0^+$, let $k_{c_n}\in\N$, $\{\xi_{n}^{1},\ldots,\xi_{n}^{k_{c_n}}\}\subset \Z^{N}\setminus\{0\}$ and $q_{c_{n},j}$ be {the solution} given by Proposition~\ref{p:pendumum} associated to $\xi_{n}^{j}$ for $j=1,\ldots, k_{c_n}$. 
Then, along a subsequence, we have that
\begin{enumerate}[label=(\roman*),start=1,itemsep=2pt]
\item\label{enum:conv1}  There exists $\kappa\in\N$ such that $k_{c_{n}}=\kappa\,$ for all $n\in\N;$
\item\label{enum:conv2} There exist distinct elements $\hat \xi^{1},\ldots,\hat \xi^{\kappa}$ in $\Z^N\setminus\{0\}$ so that 
$\{n_1\hat\xi^1+\ldots+n_{\kappa}\hat\xi^{\kappa}:n_1,\ldots,n_{\kappa}\in\Z\}=\Z^N$, and
\[
\xi_{n}^{1}=\hat\xi^{1},\ldots,\xi_{n}^{k_{c_n}}=\hat\xi^{\kappa}\;\,\text{ for all }\,n\in\N;
\]
\item\label{enum:conv3} For any $j\in\{1,\ldots,\kappa\}$ there is a solution $q^{j}\in C^{2}(\R,\R^{N})$ to~\eqref{eqn:gradSystem} of heteroclinic type between $0$ and $\hat\xi^{j}$,  and there exists $(\tau_{n,j})\subset\R$ such that
\[
q_{c_n,j}(\cdot-\tau_{n,j})\to q^{j}\;\text{ in }\;C^{2}_{loc}(\R,\R^{N}),\;\text{ as }\;n\to+\infty.
\]
\end{enumerate}

\end{prop}
\begin{proof}[{\bf Proof of Proposition~\ref{p:convergencepend}}] 
{With no loss of generality we can assume {$c_n<c_{p}$} for all $n\in\N$}. Let $k_{c_n}\in\N$, $\{\xi_{n}^{1},\ldots,{\xi_{n}^{k_{c_n}}}\}\subset \Z^{N}\setminus\{0\}$ and $q_{c_{n},j}$, for $j=1,\ldots,k_{c_n}$ be given by Proposition~\ref{p:pendumum} associated to $\xi_{n}^{j}$. 
We know $\|q_{c_n,j}\|_{L^{\infty}(\R,\R^{N})}\leq R_{p}+1$ for any $n\in\N$. Moreover $\inf_{t\in\R}\dist(q_{c_{n},j}(t),{\mathcal{V}^{c_n}_0})=\inf_{t\in\R}\dist(q_{c_{n},j}(t),\mathcal{V}^{c_n}_{\xi_{n}^{j}})=0$ for any $j\in \{1,\ldots,k_{c_n}\}$ {and all $n\in\N$}.\\ 
\noindent In particular, the above shows that $\{\xi_{n}^{1},\ldots,\xi_{n}^{k_{c_n}}\}\subset B_{R_{p}+1}(0)\cap\Z^{N}$ for any $n\in\N$. Since $B_{R_{p}+1}(0)\cap\Z^{N}$ is
a finite set, all the sequences $(k_{c_n})$, $(\xi_{n}^{j})$ for $1\leq j\leq k_{c_n}$ take their values in a finite set, hence they are all constant along a common subsequence: there exists $\kappa\in\N$, $\{\hat \xi^{1},\ldots,\hat \xi^{\kappa}\}\subset\Z^N\setminus\{0\}$ and an increasing sequence $(n_{i})\subset\N$, such that $k_{c_{n_i}}=\kappa$, $\xi_{n_{i}}^{j}=\hat\xi^{j}$ for any $i\in\N$ and $1\leq j\leq\kappa$. Thus~\ref{enum:conv1} and \ref{enum:conv2} follow. 

For $j\in\{1,\ldots,\kappa\}$ fixed, and let us {simplify} the notation by allowing 
\[
q_i:=q_{c_{n_{i}},j},\; c_{i}:=c_{n_{i}}\;\text{ and }\;
\xi:=\hat\xi^{j}=\xi_{n_{i}}^{j}\;\;{\hbox{for all $i\in\N$}}.
\]
Since {any} $q_{i}$ is given by Theorem~\ref{P:main1}, we can invoke Remark~\ref{R:min} to see that, for each $i\in\N$, $q_{i}$ has a connecting time interval $(\alpha_i,\omega_i)\subset\R$ with $-\infty\leq\alpha_i\leq\omega_i\leq+\infty$ in such a way that
\begin{enumerate}[label={\rm$(\arabic*_{i})$},itemsep=2pt]
	\item\label{enum:pend:1n} $V(q_{i}(t))>c_{i}$ for every $t\in (\alpha_{i},\omega_{i})$,
	\item\label{enum:pend:2n} $\lim\limits_{t\to\alpha_{i}^{+}}\dist(q_{i}(t),\mathcal{V}^{c_{i}}_0)=0$, and if $\alpha_{i}>-\infty$ then $\dot q_{i}(\alpha_{i})=0$, $V(q_{i}(\alpha_{i}))=c_{i}$ with $q_{i}(\alpha_{i})\in\mathcal{V}^{c_{i}}_0$,
	 \item\label{enum:pend:3n} $\lim\limits_{t\to\omega_{i}^{-}}\dist(q_{i}(t),\mathcal{V}^{c_{i}}_{\xi})=0$, and if $\omega_{i}<+\infty$ then ${\dot q_{i}}(\omega_{i})=0$, $V(q_{i}(\omega_{i}))= c_{l}$ with $q_{i}(\omega_{i})\in\mathcal{V}^{c_{i}}_{\xi}$,
	 \item\label{enum:pend:4n}  $J_{c_{i},(\alpha_{i},\omega_{i})}(q_{i})=m_{c_{i}}=\inf\limits_{q\in\Gamma_{c_{i}}}J_{c_{i}}(q)$. 
\end{enumerate}
The rest of our argument goes along the same lines as the proof convergence to a heteroclinic type solution in Proposition~\ref{p:convergence}, so we will briefly review it. We first renormalize the sequence $(q_i)$ by a phase shift procedure. From~\ref{enum:pend:1n}-\ref{enum:pend:2n}-\ref{enum:pend:3n} it follows that for all $i\in\N$ there is $t_i\in(\alpha_i,\omega_i)$ so that
$|q_i(t_i)|=\tfrac 12$.
Renaming, {if necessary}, $q_i$ to be $q_i(\cdot-t_i)$, we can assume
\begin{equation*}\label{eq:pend:trasl}
\alpha_i<0<\omega_i\;\text{ and }\;|q_i(0)|=\tfrac12,\;\text{ for any }i\in\N.
\end{equation*}

Just like before, the bound $\sup_{i\in\N}\|q_i\|_{L^{\infty}(\R,\R^N)}\leq R_{p}+1$ combined with energy constraint $E_{q_i}=-c_i$ and the fact that $q_i$ solves the system $\ddot q=\nabla V(q)$ on $\R$, allows us to conclude that $(\dot q_i)$ and $(\ddot q_i)$ are uniformly bounded in $L^{\infty}(\R,\R^N)$. Whence, by the Ascoli-Arsel\`a Theorem there is $q_0\in C^1(\R,\R^N)$ so that $(q_i)$ has a subsequence, yet denoted by $(q_i)$, for which $q_i\to q_0$ in $C^1_{loc}(\R,\R^N)$ {as $i\to +\infty$}. This convergence is then bootstrapped into the equation $\ddot q_i=\nabla V(q_i)$ in order to enhance it to $C^2_{loc}(\R,\R^N)$. 
This shows, in turn, that
\[
\ddot q_0=\nabla V(q_0)\;\text{ on }\;\R.
\]
In addition, by taking the limit $i\to+\infty$ {in $|q_i(0)|=\tfrac12$ and $\|q_i\|_{L^{\infty}(\R,\R^N)}\leq R_{p}+1$}, we learn that this solution satisfies
\begin{equation}\label{eq:pend:dist0limit}
|q_0(0)|=\tfrac12\;\hbox{ and }\;\|q_0\|_{L^{\infty}(\R,\R^N)}\leq R_{p}+1.
\end{equation}
Furthermore, the same argument used to establish~\eqref{eq:infty} can be applied to our case, to yield
\begin{equation*}\label{eq:pend:infty}
\alpha_{i}\to -\infty\;\hbox{ and }\;\omega_{i}\to+\infty,\;\text{ as }\;i\to+\infty,
\end{equation*}
for potentials satisfying~\ref{enum:pendulum:V5}-\ref{enum:pendulum:V6}: arguing by contradiction, we show that $q_0$ solves the Cauchy problem $\ddot q=\nabla q$ and $q(0)=0,\dot q(0)=0$, whence $q_0\equiv 0$, which is contrary to~\eqref{eq:pend:dist0limit}.

To conclude the proof of Proposition~\ref{p:convergenceduffing} we are left to show that $q_0$ is of heteroclinic type. More precisely, our goal is to show that
\begin{equation}\label{eq:pend:goal}
q_0(t)\to 0\;\text{ as }\;t\to-\infty,\;\text{ and }\;\;q_0(t)\to \xi\;\text{ as }\;t\to+\infty.
\end{equation}
As before, interpolation inequalities would then prove that $q_0$ is a solution to~\eqref{eqn:gradSystem} of heteroclinic type between $0$ and $\xi$. By analogy with our previous analysis,~\eqref{eq:pend:goal} reduces to proving { that for any $r\in (0,\frac13)$ there exist $L^-_r,L^+_r>0$ and $i_r\in\N$, in such a way that
\[
|q_{i}(t)|<r\;\hbox{ for }\,t\in (\alpha_{i},-L^{-}_{r})\;\;\hbox{ and }\;\;
|q_{i}(t)-\xi|<r\;\hbox{ for }\,t\in (L^{+}_{r},\omega_{i}),
\]
 for all $i\ge i_r$.}
The proof of this assertion can be obtained by rephrasing the argument used to prove {\eqref{eq:trappingqn} in} Proposition~\ref{p:convergence}, and we omit it.
\end{proof}

\vskip 3em

\end{document}